\documentclass[12pt,leqno]{amsart}
\usepackage[latin1]{inputenc}
\usepackage{amssymb}
\usepackage{enumerate}
\usepackage[all]{xy}
\usepackage{hyperref,cleveref,graphics,mathrsfs}
\usepackage{lmodern}
\usepackage{amsmath}
\usepackage{mathtools}
\usepackage{commath}
\usepackage{amsfonts}
\usepackage{dsfont}
\usepackage{graphicx}
\usepackage{amsthm}
\numberwithin{equation}{section}

\newtheorem{thm}{Theorem}[section]
\newtheorem{cor}[thm]{Corollary}
\newtheorem{lem}[thm]{Lemma}
\newtheorem{prop}[thm]{Proposition}
\theoremstyle{definition}
\newtheorem{defn}[thm]{Definition}

\newtheorem{rem}[thm]{Remark}

\numberwithin{equation}{section}

\def\epsilon{\varepsilon}
\newcommand{\N}{\mathbb{N}}

\newcommand{\R}{\mathbb{R}}
\newcommand{\C}{\mathbb{C}}
\newcommand{\uno}{\mathds{1}}
\newcommand{\Ccal}{\mathcal{C}}

\newcommand{\spn}{\mathrm{span}}
\newcommand{\lat}{\mathrm{lat}}
\newcommand{\fbl}{\mathrm{FBL}}
\newcommand{\fbp}{{\mathrm{FBL}}^{(p)}}
\newcommand{\FBLi}{\mathrm{FBL}^{(\infty)}}

\usepackage{bbm}

\title[Free dual Banach spaces and lattices]{Free dual spaces and free Banach lattices} 

\subjclass[2020]{46B42, 46B10, 46A11}

\author[E. Garc\'ia-S\'anchez]{E.~Garc\'ia-S\'anchez}
\address{Instituto de Ciencias Matem\'aticas (CSIC-UAM-UC3M-UCM)\\
Consejo Superior de Investigaciones Cient\'ificas\\
C/ Nicol\'as Cabrera, 13--15, Campus de Cantoblanco UAM\\
28049 Madrid, Spain.}
\email{enrique.garcia@icmat.es}

\author[P. Tradacete]{P.~Tradacete}
\address{Instituto de Ciencias Matem\'aticas (CSIC-UAM-UC3M-UCM)\\
Consejo Superior de Investigaciones Cient\'ificas\\
C/ Nicol\'as Cabrera, 13--15, Campus de Cantoblanco UAM\\
28049 Madrid, Spain.}
\email{pedro.tradacete@icmat.es}

\date{\today}

\keywords{Free Banach lattice; dual Banach space; dual Banach lattice; local reflexivity principle.}

\thanks{Research partially supported by grants PID2020-116398GB-I00 and CEX2019-000904-S funded by MCIN/AEI/10.13039/501100011033. E. Garc\'ia-S\'anchez was partially supported by FPI grant CEX2019-000904-S-21-3 funded by MCIN/AEI/10.13039/501100011033 and by ``European Union NextGenerationEU/PRT''. P. Tradacete is also supported by a 2022 Leonardo Grant for Researchers and Cultural Creators, BBVA Foundation}

\begin{document}

\begin{abstract}
The relation between the free Banach lattice generated by a Banach space and free dual spaces is clarified. In particular, it is shown that for every Banach space $E$ the free $p$-convex Banach lattice generated by $E^{**}$, denoted $\fbp[E^{**}]$, admits a canonical isometric lattice embedding into $\fbp[E]^{**}$ and $\fbp[E^{**}]$ is lattice finitely representable in $\fbp[E]$. Moreover, we also show that for $p>1$, $\fbp[E]^{**}$ can actually be considered as the free dual $p$-convex Banach lattice generated by $E$, whereas for $p=1$ this happens precisely when $E$ does not contain complemented copies of $\ell_1$.
\end{abstract}

\date{\today}
\maketitle

\section{Introduction}

The purpose of this note is to clarify the relation between free Banach lattices, free dual spaces and free dual Banach lattices generated by a Banach space.\\

Recall that given a Banach space $E$, the free Banach lattice generated by $E$ is a Banach lattice $\fbl[E]$, equipped with a linear isometric embedding $\phi_E:E\rightarrow \fbl[E]$ with the following universal property: for every Banach lattice $X$ and every bounded linear operator $T:E\rightarrow X$, there is a unique lattice homomorphism $\hat T:\fbl[E]\rightarrow X$ such that $\hat T \circ \phi_E=T$, with $\|\hat T\|=\|T\|$. The existence of $\fbl[E]$ together with an explicit construction was given in \cite{ART}, motivated by ideas in \cite{dePW}, and has been the object of intense research in the recent years (see for instance \cite{AMR2} for the relation with projectivity, \cite{AMRT} for lifting properties, \cite{DMRR2} for the connections with norm-attaining lattice homomorphisms, or \cite{dHT} for developments in the setting of complex scalars).\\

The construction of $\fbl[E]$ is a particular instance of the general scheme of free objects in a certain category, generated by an object in a larger category (see \cite{GHT} for a recent survey on free objects related to Banach spaces). In the case considered above, one can take the category $\mathcal{BL}$ of Banach lattices with lattice homomorphisms and the category $\mathcal{B}an$ of Banach spaces with bounded linear operators. Thus, for an object $E$ in $\mathcal{B}an$, $\fbl[E]$ is an object in $\mathcal{BL}$, and moreover, for each morphism in $\mathcal{B}an$ (that is, a bounded linear operator) $T:E\rightarrow F$ between objects from $\mathcal{B}an$, we can associate a morphism in $\mathcal{BL}$ (a lattice homomorphism) $\overline{T}:\fbl[E]\rightarrow \fbl[F]$ which is given by $\overline{T}=\widehat{\phi_F \circ T}$. Thus, this construction provides a connection between Banach spaces and Banach lattices which has proven to be a useful tool for the study of the interactions between their properties (see \cite{OTTT} for an up-to-date account on this connection).\\

In an entirely analogous manner, one can consider the category $\mathcal{B}an^*$ of dual Banach spaces with adjoint (equivalently, weak*-to-weak* continuous) operators. In this setting, it is easy to check that for a Banach space $E$, the second dual $E^{**}$ (together with the canonical embedding $J_E:E\rightarrow E^{**}$) provides the \textit{free dual Banach space} generated by $E$, as for every operator into a dual Banach space $T:E\rightarrow X^*$ there is a unique weak*-to-weak* continuous extension to $E^{**}$ given by $J_X^* \circ T^{**}$. \\

The situation with dual Banach lattices is somewhat more involved. Note first that it is an open problem whether every Banach lattice which is a dual Banach space is actually the dual of a Banach lattice (for separable dual spaces this was solved in the affirmative by Talagrand in \cite{Talagrand}). Here, we will consider the category $\mathcal{BL}^*$ of duals of Banach lattices, whose morphisms are weak*-to-weak* continuous lattice homomorphisms. It will be relevant for this to keep in mind that the adjoints of lattice homomorphisms are the (almost) interval preserving operators \cite[Theorem 1.4.19]{MeyerNieberg}. Our aim is to explore the existence of free objects in $\mathcal{BL}^*$ generated by a Banach space. Our main result (\Cref{theo: bidual of FBL is FBL*E iff ell1 not complemented in E}) is that $\fbl[E]^{**}$ can be identified with the free dual Banach lattice generated by the Banach space $E$ precisely when $E$ does not contain any complemented subspace isomorphic to $\ell_1$.\\

Motivated by this, we will also compare the free dual space of the free Banach lattice generated by a Banach space with the free Banach lattice generated by the corresponding free dual space. Namely, we will show that for every Banach space $E$, there is a canonical lattice isometric embedding of $\fbp[E^{**}]$ into $\fbp[E]^{**}$. Indeed, consider the canonical embedding $\phi_E:E\rightarrow \fbp[E]$, take its double adjoint $\phi_E^{**}:E^{**}\rightarrow \fbp[E]^{**}$, and finally, extend it to the unique lattice homomorphism $\widehat{\phi_E^{**}}:\fbp[E^{**}]\rightarrow \fbp[E]^{**}$. By means of a careful application of the principle of local reflexivity, we will prove that $\widehat{\phi_E^{**}}$ is a lattice isometric embedding. Moreover, using our technique and an argument for computing locally the free norm given in \cite{Oikhberg} we can also show that $\fbp[E^{**}]$ is lattice finitely representable in $\fbp[E]$.\\

The paper is organized as follows: Section \ref{sec:preliminares} is devoted to some preliminaries on the construction of the free $p$-convex Banach lattice $\fbp[E]$ and the local reflexivity principle. In Section \ref{sec:bidual embedding}, we show that there is a canonical lattice isometric embedding of $\fbp[E^{**}]$ into $\fbp[E]^{**}$. With similar techniques we can also show that $\fbp[E^{**}]$ is lattice finitely representable in $\fbp[E]$. In Section \ref{sec:free dual BL} we study the free dual Banach lattice generated by a Banach space. First, it is shown that for $p>1$, the free $p$-convex dual Banach lattice generated by a Banach space $E$ coincides with $\fbp[E]^{**}$. Next, we focus on the more complicated case of $p=1$, showing that this only happens when the Banach space $E$ does not contain complemented subspaces isomorphic to $\ell_1$. Finally, in Section \ref{sec:AM-spaces}, we generalize some of these results to the category $\mathcal{AL}^*$ of duals of $AL$-spaces.

\section{Preliminary computations and local reflexivity}\label{sec:preliminares}

For convenience and generality, we will state our results in the setting of free $p$-convex Banach lattices, which generalize the construction of $\fbl[E]$ (see \cite{JLTTT}).\\ 

Recall that, for $1\leq p \leq \infty$, the free $p$-convex Banach lattice over a Banach space $E$ is a $p$-convex Banach lattice $\fbp[E]$, together with a linear isometric embedding $\phi_E: E\rightarrow \fbp[E]$, such that for every bounded linear operator $T$ from $E$ to a $p$-convex Banach lattice $X$ there exists a unique lattice homomorphism $\hat{T}:\fbp[E]\rightarrow X$ such that $\hat{T}\circ \phi_E =T$, with $\norm[0]{\hat{T}}\leq M^{(p)}(X)\norm{T}$, where $M^{(p)}(X)$ denotes the $p$-convexity constant of $X$.\\

Let us also recall the explicit construction of $\fbp[E]$. Denote by $H[E]$ the set of functions $f:E^*\rightarrow~\R$ which are positively homogeneous, and for $1\leq p<\infty$ define
\begin{equation}\label{equa: norm HpE}
	\|f\|_{\fbp[E]}:=\sup \cbr{\intoo[3]{\sum_{i=1}^m \abs[0]{f(x_i^*)}^p}^{\frac{1}{p}}: \sup_{x\in B_E} \intoo[3]{\sum_{i=1}^m \abs[0]{x_i^*(x)}^p}^{\frac{1}{p}}\leq 1}
\end{equation}
where $m\in\N$ and $x_1^*,\ldots,x_m^*\in E^*$, while for $p=\infty$ set
\begin{equation}\label{equa: norm HinfE}
	\|f\|_{\FBLi[E]}:=\sup \cbr{\abs[0]{f(x^*)}: x^*\in B_{E^*}}.
\end{equation}
It is straightforward to check that the space $$H_p[E]:=\cbr{f\in H[E]: 	\|f\|_{\fbp[E]}<\infty}$$ is a Banach lattice when endowed with the pointwise order and lattice operations and the above norm. Defining $\phi_E: E \rightarrow H_p[E]$ by 
\begin{equation}\label{equa: definition phi_E}
    \phi_Ex(x^*)=x^*(x)
\end{equation} for every $x\in E$ and $x^*\in E^*$, it can be shown that $\phi_E$ is a linear isometric embedding and the closed sublattice of $H_p[E]$ generated by $\phi_E(E)$ satisfies the universal property defining the free $p$-convex Banach lattice generated by $E$ \cite[Theorem 6.1]{JLTTT}.\\

Let $E, F$ be Banach spaces and $1\leq p \leq \infty$. For a function $f:E\rightarrow F$ consider the following expression for $1\leq p <\infty$
\begin{equation}\label{eq: norm of HpEF}
    \norm{f}_{\mathcal{H}_p(E,F)}:=\sup \left\{\Big(\sum_{i=1}^m |y^*(f(x_i))|^p\Big)^{\frac{1}{p}}: \sup_{x^*\in B_{E^*}}\sum_{i=1}^m |x^*(x_i)|^p\leq 1\right\},
\end{equation}
where $y^*\in B_{F^*}$, $(x_i)_{i=1}^m\subset E$ and $m\in\mathbb N$ is arbitrary, and
\begin{equation}\label{eq: norm of HinfinityEF}
    \norm{f}_{\mathcal{H}_{\infty}(E,F)}:=\sup \left\{\norm{f(x)}_F: x\in B_E \right\}
\end{equation}
when $p=\infty$. Let us denote
\begin{align*}
    \mathcal H_p(E,F)=\{f:E\rightarrow F: & f\text{ is positively homogeneous and } \\
    &\norm{f}_{\mathcal{H}_p(E,F)}<\infty \}.
\end{align*}
These spaces, which are Banach spaces with the norm given by \eqref{eq: norm of HpEF}, are relevant in the study of lattice homomorphisms between $\fbp$ spaces (see \cite{Laust-Tra, OTTT}). When $F=\mathbb R$, we can equip $\mathcal H_p(E,\mathbb R)$ with the pointwise order and lattice operations, and it becomes a Banach lattice with the above norm. In particular, for every Banach space $E$ we have that $H_p[E]=\mathcal{H}_p(E^*,\mathbb R)$.
\medskip

Given a Banach space $E$, its dual $E^*$ embeds into $\mathcal{H}_p(E^{**},\mathbb R)$ through the isometric embedding $\phi_{E^*}$ defined by \eqref{equa: definition phi_E}. In the same way, $E^*$ can be embedded into $\mathcal{H}_p(E,\mathbb R)$ by means of the map
\begin{equation*}
    \fullfunction{\psi_{E^*}}{E^*}{\mathcal{H}_p(E,\mathbb R)}{x^*}{\psi_{E^*}x^*=x^*,}
\end{equation*}
which is again a linear isometry. Let us consider the composition operator
\begin{equation*}
    \fullfunction{C_{J_E}}{\mathcal{H}_p(E^{**},\mathbb R)}{\mathcal{H}_p(E,\mathbb R)}{h}{C_{J_E}h=h\circ J_E,}
\end{equation*}
where $J_E:E\rightarrow E^{**}$ is the canonical embedding. This operator is well defined and has norm one, since
\begin{equation}\label{equa: w-summing on E implies w-summing on E**}
    \sup_{x^*\in B_{E^*}}\intoo[3]{\sum_{i=1}^m \abs[0]{x^*(x_i)}^p}^{\frac{1}{p}} = \sup_{x^*\in B_{E^*}}\intoo[3]{\sum_{i=1}^m \abs[0]{J_Ex_i(x^*)}^p}^{\frac{1}{p}}
\end{equation}
for every $(x_i)_{i=1}^m\subset E$. It is also a lattice homomorphism, because $\mathcal{H}_p(E,\mathbb R)$ and $\mathcal{H}_p(E^{**},\mathbb R)$ are equipped with the pointwise lattice operations, and it satisfies that $C_{J_E}\circ \phi_{E^*}=\psi_{E^*}$.\\

We can also define the ``extension by zero'' operator $\kappa_{E}:\mathcal{H}_p(E,\mathbb R)\rightarrow \mathcal{H}_p(E^{**},\mathbb R)$ given for $g\in \mathcal{H}_p(E,\mathbb R)$ by
\begin{equation*}
    \kappa_{E}g(x^{**})=
    \begin{cases}
		g(J_E^{-1}x^{**}) & \text{if } x^{**}\in J_E(E),\\
		0 & \text{if } x^{**}\notin J_E(E).
	\end{cases}
\end{equation*}
    
\begin{lem}\label{prop: GpE* embeds into HpE*}
    For every Banach space $E$, $C_{J_E}:\mathcal{H}_p(E^{**},\mathbb R)\rightarrow \mathcal{H}_p(E,\mathbb R)$ is a surjective lattice homomorphism, and $\kappa_E:\mathcal{H}_p(E,\mathbb R)\rightarrow \mathcal{H}_p(E^{**},\mathbb R)$ is an isometric lattice embedding such that $C_{J_E}\circ \kappa_{E}= I_{\mathcal{H}_p(E,\mathbb R)}$. In particular, $\mathcal{H}_p(E^{**},\mathbb R)$ contains a lattice isometric copy of $\mathcal{H}_p(E,\mathbb R)$ complemented by a lattice projection.
\end{lem}

\begin{proof}
    Let $1\leq p <\infty$ (the case $p=\infty$ follows with the obvious modifications). Firstly, let us check that $\kappa_{E}$ is well defined. Given $g\in \mathcal{H}_p(E,\mathbb R)$, $\kappa_{E}g$ is always positively homogeneous, since $J_E(E)$ is a linear subspace of $E^{**}$ and $g$ is already positively homogeneous on $E$. In order to estimate $\norm[0]{\kappa_{E}g}_{\mathcal{H}_p(E^{**},\mathbb R)}$, we choose an arbitrary sequence $(x_i^{**})_{i=1}^m\subset E^{**}$, $m\in \N$, that we can assume without loss of generality to be of the form $\cbr{J_Ex_1,\ldots,J_Ex_n, x_{n+1}^{**},\ldots,x_m^{**}}$ for some $0\leq n\leq m$. It follows that
    \begin{align*}
        \intoo[3]{\sum_{i=1}^m \abs[0]{\kappa_{E}g(x_i^{**})}^p}^{\frac{1}{p}} & = \intoo[3]{\sum_{i=1}^n \abs[0]{\kappa_{E}g(J_Ex_i)}^p}^{\frac{1}{p}} = \intoo[3]{\sum_{i=1}^n \abs[0]{g(x_i)}^p}^{\frac{1}{p}}\\
        & \leq \norm[0]{g}_{\mathcal{H}_p(E,\mathbb R)}\sup_{x^*\in B_{E^*}}\intoo[3]{\sum_{i=1}^n \abs[0]{x^*(x_i)}^p}^{\frac{1}{p}}\\
        &= \norm[0]{g}_{\mathcal{H}_p(E,\mathbb R)}\sup_{x^*\in B_{E^*}}\intoo[3]{\sum_{i=1}^n \abs[0]{J_Ex_i(x^*)}^p}^{\frac{1}{p}}\\
        & \leq \norm[0]{g}_{\mathcal{H}_p(E,\mathbb R)} \sup_{x^*\in B_{E^*}}\intoo[3]{\sum_{i=1}^m \abs[0]{x_i^{**}(x^*)}^p}^{\frac{1}{p}},
    \end{align*}
    so $\norm[0]{\kappa_E g}_{\mathcal{H}_p(E^{**},\mathbb R)}\le \norm[0]{g}_{\mathcal{H}_p(E,\mathbb R)}$ and $\kappa_{E}$ is well defined and bounded. On the other hand, given a sequence $(x_i)_{i=1}^m\subset E$ we observe that
    \begin{align*}
        \intoo[3]{\sum_{i=1}^m \abs[0]{g(x_i)}^p}^{\frac{1}{p}} & = \intoo[3]{\sum_{i=1}^m \abs[0]{\kappa_{E}g(J_Ex_i)}^p}^{\frac{1}{p}} \\
        & \leq \norm[0]{\kappa_E g}_{\mathcal{H}_p(E^{**},\mathbb R)} \sup_{x^*\in B_{E^*}}\intoo[3]{\sum_{i=1}^m \abs[0]{J_Ex_i(x^*)}^p}^{\frac{1}{p}} \\
        & = \norm[0]{\kappa_E g}_{\mathcal{H}_p(E^{**},\mathbb R)}  \sup_{x^*\in B_{E^*}}\intoo[3]{\sum_{i=1}^m \abs[0]{x^*(x_i)}^p}^{\frac{1}{p}},
    \end{align*}
    so we can conclude that $\kappa_{E}$ is an isometry. It can be easily checked that $\kappa_{E}$ is a lattice homomorphism and that $C_{J_E}\circ \kappa_{E}= I_{\mathcal{H}_p(E,\mathbb R)} $.
\end{proof}

\begin{rem}\label{rema: lambda E* not injective}
    Note that in general $C_{J_E}$ is not injective. For instance, suppose that $E$ is not reflexive. By the Hahn--Banach Theorem there exists a non-zero functional $x^{***}\in E^{***}$ such that $x^{***}=0$ on $J_E(E)$. Since $E^{***}$ embeds canonically into $\mathcal{H}_p(E^{**},\mathbb R)$, we can consider $C_{J_E}x^{***}=x^{***}\circ J_{E}$, which is identically zero by construction.
\end{rem}

Nevertheless, we can consider $\fbp[E^*]$ as a sublattice of $\mathcal{H}_p(E^{**},\mathbb R)$, and we will see in \Cref{theo: FBL of bidual embeds into bidual of FBL} that when restricting $C_{J_E}$ to $\fbp[E^*]$ it turns out that it is an isometric embedding. For the proof of this, we need first to recall the Principle of Local Reflexivity (cf. \cite[Theorem 6.3]{FHHMZBanachSpaceTheory}) and make a careful analysis of the approximation parameters involved in the version given next.

\begin{thm}\label{theo: principle of local reflexivity}
	Let $X$ be a Banach space. For any finite-dimensional subspaces $U\subset X^{**}$ and $V\subset X^*$ and $\epsilon>0$, there exists a linear isomorphism $S$ of $U$ onto $S(U)\subset X$ such that $\norm{S}\norm[0]{S^{-1}}\leq 1+\epsilon$, $x^*(Sx^{**})=x^{**}(x^*)$ for every $x^*\in V$ and $x^{**}\in U$, and $S$ is the identity on $U\cap J_X(X)$.
\end{thm}

In the standard proof of this result, in order to construct the operator $S$ (see \cite[Theorem 6.3]{FHHMZBanachSpaceTheory} for instance) a $\delta$-net $\cbr{b_1^{**},\ldots, b_n^{**}}$ of the unit sphere $S_U$ satisfying some additional properties is needed, with $\delta$ chosen in such a way that
\begin{equation*}
	\theta(\delta):= \frac{1+\delta}{1-\delta} \left(\frac{1}{1+\delta}-\delta \frac{1+\delta}{1-\delta} \right)^{-1}<1+\epsilon.
\end{equation*}
The operator $S$ can then be defined, so that
\begin{equation*}
	(1+\delta)^{-1}\leq \norm[0]{Sb_i^{**}} \leq 1+\delta
\end{equation*}
for each $i=1,\ldots,n$. Under these conditions, it is straightforward to check (see \cite[Ex. 1.76]{FHHMZBanachSpaceTheory}) that 
\begin{equation}\label{equa: estimate norm S local reflexivity}
	\norm{S}\leq \frac{1+\delta}{1-\delta}\,\,\,\text{and}\,\,\, \norm[0]{S^{-1}}\leq \left(\frac{1}{1+\delta}-\delta \frac{1+\delta}{1-\delta} \right)^{-1}.
\end{equation}
In particular, $S$ is invertible and $\norm{S}\norm[0]{S^{-1}}\leq 1+\epsilon$. The above estimates on the norm of $S$ and its inverse lead us to study the behaviour of $\delta\equiv \delta(\epsilon)$ when $\epsilon$ is small. The following technical lemma shows more precisely how this dependence on $\epsilon$ can be described.

\begin{lem}\label{lemm: technical delta}
	There exists some $\epsilon_0>0$ such that for any $\epsilon \in (0,\epsilon_0)$ there is a $\delta\equiv \delta(\epsilon)\in \intoo[1]{0,\frac{\epsilon}{1+\epsilon}}$ such that $\theta(\delta)=1+\epsilon$.
\end{lem}

\begin{proof}
	Putting $z=1+\delta$ and $\lambda\equiv \lambda(\epsilon)= \frac{\epsilon}{1+\epsilon}$ we can transform the equation $\theta(\delta)=1+\epsilon$ into $g(z)=0$, where $g(z)=z^3-\lambda z^2+z-2$. Our goal is to show that, for every $\epsilon$ small enough, $g$ has a real root $z_0\equiv z_0(\epsilon)$ such that $1<z_0<1+\lambda$. To do so, we will make use of Rouch\'e's Theorem (cf. \cite[Section VIII.2]{Gamelin}) to compare the polynomial $g(z)$ with $f(z)=z^3+z-2= (z-1)(z^2+z+2)$, which has a simple root when $z=1$ and two complex-conjugate roots in the half-plane $\Re z<0$. \\
	
	Let us consider the auxiliary function $$\phi_{\lambda}(r)=-r^3 -(3+\lambda)r^2 +(4-2\lambda)r-\lambda.$$ 
	Given $r\in (0,1)$, if $\phi_{\lambda}(r)>0$, then $\abs{g(z)-f(z)}<\abs{f(z)}$ for all $z\in \partial D(1,r)$, where $D(1,r)$ is the open disc of radius $r$ centered at 1, and $\partial D(1,r)$ is its boundary. Indeed, $\phi_{\lambda}(r)>0$ if and only if $$\lambda (1+r)^2<r(4-3r-r^2).$$ In this situation,
    \begin{align*}
        \abs{g(z)-f(z)}& =\lambda\abs{z}^2\leq \lambda (1+r)^2<r(4-3r-r^2)\\
        & \leq \abs{z-1}\abs[0]{z^2+z+2}=\abs{f(z)}
    \end{align*}
	for all $z\in \partial D(1,r)$, where the last inequality is obtained as follows: since $\Re z\geq 1-r>0$ for every $z\in \partial D(1,r)$, we have that
	\begin{eqnarray*}
		\abs[0]{z^2+z+2}&=& \abs[0]{(z-1)^2+3z+1}\geq \abs{3z+1}-\abs{z-1}^2\\&\geq&\abs{\Re(3z+1)}-r^2\geq 4-3r-r^2.
	\end{eqnarray*}
    Now, evaluating at $r=\lambda$ we get that $\phi_{\lambda}(\lambda)=3\lambda-5\lambda^2-2\lambda^3$, which is positive whenever $\lambda$ is positive and small enough. Let $\epsilon_0\in (0,1)$ be such that $\phi_{\lambda} (\lambda)>0$ for every $\lambda\in (0,\epsilon_0)$. Then, for any $\epsilon\in (0,\epsilon_0)$ the radius $r= \lambda(\epsilon)\in (0,\epsilon_0)$ satisfies $\phi_{\lambda}(r)>0$, so $\abs{g(z)-f(z)}<\abs{f(z)}$ for every $z\in \partial D(1,r)$. Since $f$ has only a simple root in $D(1,r)$, Rouch\'e's Theorem allows us to conclude that $g$ has a single root in $D(1,r)$ as well, denoted $z_0$.\\
	
	Such $z_0$ must be real, because if $z_0\in \C\setminus \R$, then its conjugate $\bar{z}_0$ would be a different root of $g$ in $D(1,r)$, which leads to a contradiction. Thus, $z_0\in (1-r,1+r)$. In addition, $z_0>1$: if $z_0\leq 1$, then we would have $0=g(z_0)=z_0^3-\lambda z_0^2+z_0-2\leq -\lambda z_0^2<0$, which is also a contradiction. Taking $\delta(\epsilon)=z_0-1>0$ we conclude the result.
\end{proof}

In particular, from \Cref{lemm: technical delta} it follows that in the proof of \cite[Theorem 6.3]{FHHMZBanachSpaceTheory} $\delta$ can always be chosen so that $\delta\rightarrow 0$ whenever $\epsilon \rightarrow 0$. This will be necessary for the results in the next section.

\section{$\fbp[E^{**}]$ embeds in $\fbp[E]^{**}$}\label{sec:bidual embedding}

As is customary, for a set $A\subset X$ in a Banach lattice, $\lat(A)$ denotes the smallest (not necessarily closed) sublattice of $X$ containing $A$. It is standard to check that $\lat(A)$ can be described as the set of expressions that use lattice and linear operations involving finitely many elements of $A$. By \cite[Exercise 8, p. 204]{AliprantisBurkinshaw}, every $f\in \lat(A)$ can be written as $$f=\bigvee_{j=1}^r \sum_{i\in A_j} \alpha^j_i a^j_i- \bigvee_{j=1}^r \sum_{i\in B_j} \beta^j_i b^j_i$$ for some $r\in\mathbb N$, finite sets $A_j,B_j\subset \mathbb N$, scalars $\alpha^j_i,\beta^j_i$ and $a^j_i,b^j_i\in A$. Let us begin with the following.

\begin{prop}\label{prop: lambda E* restricted to FBL is embedding}
    Let $1\leq p \leq \infty$. The restriction of the operator $C_{J_E}$ to $\fbp[E^*]$ is an isometric lattice embedding.
\end{prop}

\begin{proof}
    Let $1\leq p <\infty$. First, we will show that $\norm[0]{h}_{\fbp[E^*]}=\|C_{J_E}h\|_{\mathcal{H}_p(E,\mathbb R)}$ for every $h\in \lat(\phi_{E^*}(E^*))$. The inequality $\norm[0]{h}_{\fbp[E^*]}\leq\|C_{J_E}h\|_{\mathcal{H}_p(E,\mathbb R)}$ follows from $\|C_{J_E}\|\leq 1$. In order to obtain the converse inequality, we apply \Cref{theo: principle of local reflexivity} to the space $X=E$. On the one hand, given $(x_i^{**})_{i=1}^m \subset E^{**}$, let us consider the finite-dimensional subspace $U=\spn(\cbr[0]{x_1^{**},\ldots,x_m^{**}})\subset E^{**}$. On the other hand, we can express a fixed $h\in \lat(\phi_{E^*}(E^*))$ as
	\begin{equation*}
		h=\bigvee_{j=1}^r \phi_{E^*} x_j^*- \bigvee_{k=1}^r \phi_{E^*} y_k^*
	\end{equation*}
	for some $(x_j^*)_{j=1}^r, (y_k^*)_{k=1}^r \subset E^*$. Let $V=\spn(\cbr{x_1^*,y_1^*,\ldots,x_r^*,y_r^*})\subset E^*$. Under these circumstances, \Cref{theo: principle of local reflexivity} yields that for every $\epsilon >0$ there exists an isomorphism $S:U\rightarrow S(U)\subset E$ such that $\norm{S}\norm[0]{S^{-1}}\leq 1+\epsilon$ and $x^*(Sx^{**})=x^{**}(x^*)$ for every $x^*\in V$ and $x^{**}\in U$. From this last property it follows that
    \begin{align*}
        C_{J_E}h(Sx_i^{**}) & = \bigvee_{j=1}^r x_j^* (Sx_i^{**})- \bigvee_{k=1}^r y_k^*(Sx_i^{**})\\
        & = \bigvee_{j=1}^r x_i^{**} (x_j^*)- \bigvee_{k=1}^r x_i^{**}(y^*_k) = h(x_i^{**}).
    \end{align*}
	Moreover,
	\begin{align*}
		\sup_{x^*\in B_{E^*}} \intoo[3]{\sum_{i=1}^m \abs[0]{x^*(Sx_i^{**})}^p}^{\frac{1}{p}}& = \sup_{x^*\in B_{E^*}} \intoo[3]{\sum_{i=1}^m \abs[0]{S^*x^*(x_i^{**})}^p}^{\frac{1}{p}} \\
		& \leq \norm{S^*} \sup_{x^{***}\in B_{E^{***}}} \intoo[3]{\sum_{i=1}^m \abs[0]{x^{***}(x_i^{**})}^p}^{\frac{1}{p}}\\ 
		& =  \norm{S} \sup_{x^*\in B_{E^*}} \intoo[3]{\sum_{i=1}^m \abs[0]{x_i^{**}(x^*)}^p}^{\frac{1}{p}}.
	\end{align*}
	Thus, if we choose $(x_i^{**})_{i=1}^m \subset E^{**}$ satisfying
	\begin{equation*}
	    \sup_{x^*\in B_{E^*}} \intoo[3]{\sum_{i=1}^m \abs[0]{x_i^{**}(x^*)}^p}^{\frac{1}{p}}\leq 1,
	\end{equation*}
	it follows that
	\begin{equation*}
	    \intoo[3]{\sum_{i=1}^m \abs[0]{h(x_i^{**})}^p}^{\frac{1}{p}}  = \norm{S} \intoo[3]{\sum_{i=1}^m \abs[2]{C_{J_E}h\intoo[2]{\frac{Sx_i^{**}}{\norm{S}}}  }^p}^{\frac{1}{p}} \leq \frac{1+\delta}{1-\delta} \|C_{J_E}h\|_{\mathcal{H}_p(E,\mathbb R)}
	\end{equation*}
    for every $\epsilon>0$ small enough, where $\delta$ has been chosen according to \Cref{lemm: technical delta}. Hence, first letting $\epsilon\rightarrow 0$, and then taking the supremum over all the possible choices of functionals $(x_i^{**})_{i=1}^m \subset E^{**}$, $m\in \N$, satisfying the $p$-summability condition above, we conclude that $\norm[0]{h}_{\fbp[E^*]}=\|C_{J_E}h\|_{\mathcal{H}_p(E,\mathbb R)}$ for every $h\in \lat(\phi_{E^*}(E^*))$. The result now follows from the fact that $\lat(\phi_{E^*}(E^*))$ is dense in $\fbp[E^*]$.\\
    
    The case $p=\infty$ is straightforward, since every function $h\in \FBLi[E^{*}]$ is weak* continuous in $B_{E^{**}}$ \cite[Theorem 5.3]{JLTTT} and, by Goldstine's Theorem (cf. \cite[Theorem 3.96]{FHHMZBanachSpaceTheory}), $J_E(B_E)$ is weak* dense in $B_{E^{**}}$.
\end{proof}

\begin{rem}\label{rema: isometric copy of FBLpE* in HpE*}
    Note that $R=\kappa_{E}\circ \eval[0]{C_{J_E}}_{\fbp[E^*]}$ is an isometric lattice embedding of $\fbp[E^*]$ into $H_p[E^*]=\mathcal{H}_p(E^{**},\mathbb R)$. When $E$ is not reflexive, the image of this embedding is different from the canonical copy of $\fbp[E^*]$. Indeed, by construction every element of $R(\fbp[E^*])$ vanishes outside of $J_E(E)$. Nevertheless, given $x^{**}\in E^{**}\setminus J_E(E)$, there is always a functional $x^*\in E^*$ such that $\phi_{E^*}x^*(x^{**})= x^{**}(x^*)\neq 0$, so $\phi_{E^*}x^*\in \fbp[E^*]$ but $\phi_{E^*}x^*\notin R(\fbp[E^*])$ and hence
    \begin{equation*}
        \fbp[E^*]\neq R(\fbp[E^*]).
    \end{equation*}
\end{rem}

\medskip

As a consequence of \Cref{prop: lambda E* restricted to FBL is embedding} we will show that $\fbp[E^{**}]$ embeds canonically into $\fbp[E]^{**}$ for any $1\leq p\leq \infty$. Before the proof of this assertion, we need some observations. First, note that \cite[Corollary 2.7]{ART} also holds for $\fbp[E]$, $1\leq p\leq \infty$, i.e. the lattice homomorphisms from $\fbp[E]$ to $\R$ (or, equivalently, the atoms of $\fbp[E]^*$) are precisely the extensions $\widehat{x^*}$ for $x^*\in E^*$.\\

Now, consider the embedding $\phi_E:E\rightarrow \fbp[E]$ and its second adjoint $\phi_E^{**}: E^{**}\rightarrow \fbp[E]^{**}$. Since $\fbp[E]^{**}$ is a $p$-convex Banach lattice with $p$-convexity constant 1 \cite[Proposition 1.d.4]{LindenstraussTzafririVol2}, it follows that $\widehat{\phi_E^{**}}:\fbp[E^{**}]\rightarrow \fbp[E]^{**} $ is a lattice homorphism with $\norm[0]{\widehat{\phi_E^{**}}}=1$. This lattice homomorphism satisfies the following property.

\begin{lem}\label{lemm: equality hat phi**E evaluated in hat x*}
	Let $1\leq p\leq \infty$. For every $h\in \fbp[E^{**}]$ and $x^*\in E^*$
	\begin{equation*}
		\widehat{\phi_E^{**}}h(\widehat{x^*})=C_{J_{E^{*}}}h (x^*).
	\end{equation*} 
\end{lem}

\begin{proof}
	Given $x^*\in E^*$, the functional $\psi_{x^*}=J_{\fbp[E]^*} \widehat{x^*}\in \fbp[E]^{***}$ is a lattice homomorphism.
	Hence, $\psi_{x^*}\circ \widehat{\phi_E^{**}}$ is a lattice homomorphism between $\fbp[E^{**}]$ and $\R$ which for $x^{**}\in E^{**}$ satisfies
	\begin{align*}
		\psi_{x^*}\circ \widehat{\phi_E^{**}} \circ\phi_{E^{**}}(x^{**})&= \psi_{x^*}\circ \phi_E^{**}(x^{**}) =\phi_E^{**}x^{**}(\widehat{x^*}) \\
		&= x^{**}(\phi_E^*\widehat{x^*}) = x^{**} (x^*)= J_{E^*}x^*(x^{**}).
	\end{align*}
    Hence, $\psi_{x^*}\circ \widehat{\phi_E^{**}}  \circ \phi_{E^{**}} = J_{E^*}x^*$. The uniqueness of extension allows us to conclude that $\psi_{x^*}\circ \widehat{\phi_E^{**}} = \widehat{J_{E^*}x^*}$. Thus, for every $h\in \fbp[E^{**}]$ we have
	\begin{equation*}
		\widehat{\phi_E^{**}} h(\widehat{x^*}) = \psi_{x^*} \circ \widehat{\phi_E^{**}}(h) =\widehat{J_{E^*}x^*}(h)= h(J_{E^*}x^*)=C_{J_{E^{*}}}h(x^*),
	\end{equation*}
	as claimed. 
\end{proof}

This property, together with \Cref{prop: lambda E* restricted to FBL is embedding}, allows us to show the next result.

\begin{thm}\label{theo: FBL of bidual embeds into bidual of FBL}
	Let $E$ be a Banach space and $1\leq p\leq \infty$. The canonical extension $\widehat{\phi_E^{**}}:\fbp[E^{**}]\rightarrow \fbp[E]^{**}$ is an isometric lattice embedding.
\end{thm}

\begin{proof}
	We begin by proving the case $p<\infty$. First, recall that $\widehat{\phi_E^{**}}$ has norm one, so for every $h\in \fbp[E^{**}]$ we have
	\begin{equation*}
		\norm[0]{\widehat{\phi_E^{**}}h}\leq  \norm[0]{h}_{\fbp[E^{**}]}.
	\end{equation*}
	To obtain the reverse inequality, let us fix a function $h\in \fbp[E^{**}]$ and $(x_i^*)_{i=1}^m\subset E^*$ such that
	\begin{equation*}
	    \sup_{x\in B_{E}} \intoo[3]{\sum_{i=1}^m \abs[0]{x_i^{*}(x)}^p}^{\frac{1}{p}}\leq 1.
	\end{equation*}
	We claim that
	\begin{equation}\label{equa: auxiliar in FBL of bidual embeds into bidual of FBL}
		\intoo[3]{\sum_{i=1}^m \abs[0]{C_{J_{E^*}}h(x_i^*)}^p}^{\frac{1}{p}} \leq \norm[0]{\widehat{\phi_E^{**}}h}.
	\end{equation}
	Indeed, by \Cref{lemm: equality hat phi**E evaluated in hat x*} we have that
	\begin{align*}
		\intoo[3]{\sum_{i=1}^m \abs[0]{C_{J_{E^*}}h(x_i^*)}^p}^{\frac{1}{p}} & = \intoo[3]{\sum_{i=1}^m \abs[0]{\widehat{\phi_E^{**}}h(\widehat{x_i^*})}^p}^{\frac{1}{p}} = \sum_{i=1}^m \theta_i \widehat{\phi_E^{**}}h(\widehat{x_i^*}) \\
		&= \widehat{\phi_E^{**}}h \intoo[3]{\sum_{i=1}^m  \theta_i \widehat{x_i^*}} \leq \norm[0]{\widehat{\phi_E^{**}}h}\norm[3]{\sum_{i=1}^m  \theta_i \widehat{x_i^*}},
	\end{align*}
	 where the coefficients $\theta_i$ are defined as
	\begin{equation*}
		\theta_i= \text{sgn}(\widehat{\phi_E^{**}}h(\widehat{x_i^*})) \abs[0]{\widehat{\phi_E^{**}}h(\widehat{x_i^*})}^{p-1} \intoo[3]{\sum_{i=1}^m \abs[0]{\widehat{\phi_E^{**}}h(\widehat{x_i^*})}^p}^{\frac{1}{p}-1}.
	\end{equation*}
	In order to estimate the norm of $\sum_{i=1}^m  \theta_i \widehat{x_i^*}\in \fbp[E]^*$, let us evaluate it on $f\in \fbp[E]$ and then apply H\"older's inequality with exponents $p$ and $p^*=\frac{p}{p-1}$ as follows:
	\begin{align*}
        & \abs[3]{\sum_{i=1}^m  \theta_i \widehat{x_i^*}(f)}  \leq  \sum_{i=1}^m  \abs[0]{\theta_i f(x_i^*)}\\
        & = \intoo[3]{\sum_{i=1}^m \abs[0]{\widehat{\phi_E^{**}}h(\widehat{x_i^*})}^p}^{\frac{1}{p}-1} \sum_{i=1}^m  \abs[0]{\widehat{\phi_E^{**}}h(\widehat{x_i^*})}^{p-1} \abs[0]{ f(x_i^*)}\\
		& \leq  \intoo[3]{\sum_{i=1}^m \abs[0]{\widehat{\phi_E^{**}}h(\widehat{x_i^*})}^p}^{\frac{1}{p}-1}  \intoo[3]{\sum_{i=1}^m \abs[0]{\widehat{\phi_E^{**}}h(\widehat{x_i^*})}^{p^*(p-1)}}^{\frac{1}{p^*}}   \intoo[3]{\sum_{i=1}^m \abs[0]{f(x_i^*)}^p}^{\frac{1}{p}} \\
		& = \intoo[3]{\sum_{i=1}^m \abs[0]{f(x_i^*)}^p}^{\frac{1}{p}} \leq 	\|f\|_{\fbp[E]}.
	\end{align*}
	Therefore, we have that $\norm[0]{\sum_{i=1}^m  \theta_i \widehat{x_i^*}}\leq 1$, and inequality \eqref{equa: auxiliar in FBL of bidual embeds into bidual of FBL} follows. 
	
	By \Cref{prop: lambda E* restricted to FBL is embedding}, we conclude that 
	$$\norm{h}_{\fbp[E^{**}]}=\norm[0]{C_{J_{E^*}}h}_{H_{p}[E]}\leq \norm[0]{\widehat{\phi_E^{**}}h}$$ for every $h\in \fbp[E^{**}]$, so $\widehat{\phi_E^{**}}$ is an isometric lattice embedding.\\

    In order to prove the case $p=\infty$, just note that by \Cref{prop: lambda E* restricted to FBL is embedding} and \Cref{lemm: equality hat phi**E evaluated in hat x*} we have
    \begin{align*}
        \norm[0]{\widehat{\phi_E^{**}}h}& \leq  \norm[0]{h}_{\FBLi[E^{**}]} = \norm[0]{C_{J_{E^*}}h}_{H_{\infty}[E]} = \sup_{x^*\in B_{E^*}} \abs[0]{C_{J_{E^*}}h(x^*)} \\ 
        & = \sup_{x^*\in B_{E^*}} \abs[0]{\widehat{\phi_E^{**}}h(\widehat{x^*})} \leq \norm[0]{\widehat{\phi_E^{**}}h}.
    \end{align*}
\end{proof}

\begin{rem}
Note that the map $\widehat{\phi_E^{**}}$ considered in \Cref{theo: FBL of bidual embeds into bidual of FBL} is not surjective as long as $\text{dim}(E)>1$. Indeed, this can be seen as a consequence of the fact that $\fbp[E]$ is not Dedekind complete \cite[Proposition 2.11]{OTTT}.
\end{rem}

\begin{rem}
It can be checked that the canonical embedding $J_{\fbp[E]}:\fbp[E]\rightarrow \fbp[E]^{**}$ can be factored as $J_{\fbp[E]}=\widehat{\phi_E^{**}}\circ \overline{J_E} $. Indeed, for every $x\in E$ we have that
$$
\widehat{\phi_E^{**}}\circ \overline{J_E} \circ \phi_E(x)=\widehat{\phi_E^{**}}\circ  \phi_{E^{**}}\circ J_E(x)=\phi_E^{**}\circ J_E(x)=J_{\fbp[E]}\circ \phi_E(x).
$$
We have that $J_{\fbp[E]}$ and $\widehat{\phi_E^{**}}\circ \overline{J_E}$ are lattice homomorphisms which coincide on points of the form $\phi_E x$, and thus they coincide on all of $\fbp[E]$. This immediately yields that $\widehat{\phi_E^{**}}$ is an isometry on the image of $\overline{J_E}$. Hence, \Cref{theo: FBL of bidual embeds into bidual of FBL} can be considered as a strengthening of this simple fact.
\end{rem}

The following corollary is just a direct consequence of \Cref{theo: FBL of bidual embeds into bidual of FBL} and the Principle of Local Reflexivity.

\begin{cor}\label{coro: FBL of bidual finitely representable in FBL}
    Let $E$ be a Banach space and $1\leq p\leq \infty$. Then $\fbp[E^{**}]$ is finitely representable in $\fbp[E]$, i.e., for every finite dimensional subspace $X\subset \fbp[E^{**}]$ and every $\epsilon>0$ there is a linear isomorphism $T$ of $X$ onto $T(X)\subset \fbp[E]$ such that $\norm{T}\norm[0]{T^{-1}}\leq 1+\epsilon$.
\end{cor}

Moreover, \Cref{prop: lambda E* restricted to FBL is embedding} together with an argument for computing the norm locally from \cite{Oikhberg} allows us to prove a lattice version of this result. This answers a natural question raised by an anonymous referee.

\begin{thm}\label{theo: FBL of bidual lattice finitely representable in FBL}
    Let $E$ be a Banach space and $1\leq p\leq \infty$. Then $\fbp[E^{**}]$ is lattice finitely representable in $\fbp[E]$, i.e., for every finite dimensional sublattice $X\subset \fbp[E^{**}]$ and every $\epsilon>0$ there is a lattice isomorphism $T$ of $X$ onto $T(X)\subset \fbp[E]$ such that $\norm{T}\norm[0]{T^{-1}}\leq 1+\epsilon$.
\end{thm}

\begin{proof}
    Let $X\subset \fbp[E^{**}]$ be a finite dimensional sublattice. We can find $f_1,\ldots,f_N\in X$ positive, normalized and pairwise disjoint elements that span $X$.\\

    {\it Claim:} Given a Banach space $F$ and $f_1,\ldots,f_N\in \fbp[F]$ positive and pairwise disjoint vectors, for every $\delta>0$ there exist $g_1,\ldots,g_N \in \lat(\phi_{F}(F))$ positive and pairwise disjoint such that $\norm{f_n-g_n}_{\fbp[F]}<\delta$ for every $n=1,\ldots,N$.\\

    We will prove this claim by induction on $N$. The case $N=1$ is trivial, since $\lat(\phi_{F}(F))_+$ is dense in $\fbp[F]_+$. Now, assume the claim holds for $N\geq 1$ vectors, and let us show that it is also true for $f_1,\ldots,f_{N+1}\in \fbp[F]$ positive and pairwise disjoint vectors. Apply the induction hypothesis to find $h_1,\ldots, h_N,h_{N+1}\in \lat(\phi_{F}(F))_+$ such that $h_1,\ldots, h_N$ are pairwise disjoint and $\norm{f_n-h_n}_{\fbp[F]}<\frac{\delta}{2N+1}$ for every $n=1,\ldots,N+1$. Define $g_n=h_n-h_n\wedge h_{N+1}$ for $n=1,\ldots,N$ and $g_{N+1}=h_{N+1}-h\wedge h_{N+1}$, where $h=\sum_{n=1}^N h_n$. It is easy to check that the vectors $g_1,\ldots,g_{N+1}$ are positive and pairwise disjoint. Moreover, we have the following inequalities:
    \begin{align*}
        \abs{h_n\wedge h_{N+1}} & \leq \abs{f_n\wedge h_{N+1}-h_n\wedge h_{N+1}} + \abs{f_n\wedge f_{N+1}-f_n\wedge h_{N+1}}\\
        & \leq \abs{f_n-h_n} + \abs{f_{N+1}-h_{N+1}}, n=1,\ldots,N,\\
        \abs{f_n-g_n} & \leq \abs{f_n-h_n} + \abs{h_n\wedge h_{N+1}}\\
        & \leq 2\abs{f_n-h_n} + \abs{f_{N+1}-h_{N+1}}, n=1,\ldots,N,
    \end{align*}
    \begin{align*}
        \abs{h\wedge h_{N+1}} & \leq \sum_{n=1}^N \abs{h_n\wedge h_{N+1}} \leq \sum_{n=1}^N \abs{f_n-h_n} + N \abs{f_{N+1}-h_{N+1}},\\
        \abs{f_{N+1}-g_{N+1}} & \leq \abs{f_{N+1}-h_{N+1}} +\abs{h\wedge h_{N+1}} \\
        & \leq \sum_{n=1}^N \abs{f_n-h_n} + (N+1) \abs{f_{N+1}-h_{N+1}}.
    \end{align*}
    Taking norms we obtain the claim.\\

    Now, let us fix $\epsilon>0$ and consider $\delta=\delta(\epsilon)>0$ to be chosen later. We can apply the claim to obtain $g_1,\ldots,g_N \in \lat(\phi_{E^{**}}(E^{**}))$ positive and pairwise disjoint such that $\norm{f_n-g_n}_{\fbp[E^{**}]}<\delta$ for every $n=1,\ldots,N$. Denote by $Y$ the subspace generated by this elements, which is already a sublattice. We can define a surjective lattice isomorphism $R:X\rightarrow Y$ given by $$R\intoo[3]{\sum_{n=1}^N \lambda_n f_n}=\sum_{n=1}^N \lambda_n g_n.$$ 
    Choosing $\delta$ small enough we can make
    \begin{equation*}
        \norm{R}\norm[0]{R^{-1}}\leq 1+\varepsilon.
    \end{equation*}
    Thus, it suffices to show that $Y$ can be lattice embedded into $\fbp[E]$ with distortion arbitrarily close to 1. Since $g_1,\ldots,g_N \in  \lat(\phi_{E^{**}}(E^{**}))$, we can find a finite dimensional subspace $F\subset E^{**}$ such that any element of $Y$ can be expressed as a lattice-linear expression with elements from $F$. By \cite[Proposition 4.4 and Remark 4.7]{Oikhberg} we can find a finite dimensional subspace $G\subset E^{**}$ containing $F$ such that 
    \begin{equation*}
        \norm{h}_{\fbp[E^{**}]}\leq \norm{h}_{\fbp[G]} \leq (1+\delta) \norm{h}_{\fbp[E^{**}]}
    \end{equation*}
    for every $h\in Y$. Now, consider a $\delta$-net $(h_j)_{j=1}^J$ in the compact set $S_Y=Y\cap S_{\fbp[E^{**}]}$. We can use \Cref{prop: lambda E* restricted to FBL is embedding} to choose functionals $x^*_{j,1},\ldots, x^*_{j,K}\in E^*$ such that
    \begin{equation*}
	    \sup_{x\in B_{E}} \intoo[3]{\sum_{k=1}^{K} \abs[0]{x_{j,k}^{*}(x)}^p}^{\frac{1}{p}}\leq 1.
	\end{equation*}
    and
    \begin{equation*}
        1= \norm{h_j}_{\fbp[E^{**}]}\leq \intoo[3]{\sum_{k=1}^{K} \abs[0]{h_j\circ J_{E^*}(x_{j,k}^{*})}^p}^{\frac{1}{p}}+\delta
    \end{equation*}
    for every $j=1,\ldots, J$. Let $V$ be the finite dimensional subspace of $E^*$ generated by these $(x^*_{j,k})$ and apply \Cref{theo: principle of local reflexivity} with $G$ and $V$ to obtain a linear isomorphism $S:G\rightarrow E$ such that $\norm{S}\norm[0]{S^{-1}}\leq 1+\delta$ and $x^*(Sx^{**})=x^{**}(x^*)$ for every $x^*\in V$ and $x^{**}\in G$. Moreover, according to \Cref{lemm: technical delta} we can assume that 
    \begin{equation*}
        \norm{S}\leq \frac{1+\delta}{1-\delta}.
    \end{equation*}
    Let $\overline S: \fbp[G]\rightarrow \fbp[E]$ be the extension of $\widehat{\phi_E\circ S}$ to a lattice homomorphism. Then, we claim that $Q=\eval[1]{\overline S}_Y: (Y, \norm{\cdot}_{\fbp[E^{**}]})\rightarrow \fbp[E]$ is a lattice embedding with distortion close to $1$ for $\delta$ small enough. To check this, first note that $\norm[0]{\overline S}=\norm{S}$ and
    \begin{equation*}
        \norm[0]{\overline S h}_{\fbp[E]}\leq \norm[0]{\overline S} \norm{ h}_{\fbp[G]}\leq  (1+\delta) \norm{S} \norm{ h}_{\fbp[E^{**}]}
    \end{equation*}
    for every $h\in Y$, so $Q$ is bounded. On the other hand, given $h\in Y$ and $x^*\in V$ it is straightforward to check that $\overline S h(x^*)=h\circ J_{E^*}(x^*)$. In particular, for every $j=1,\ldots, J$ we have
    \begin{align*}
        \norm[0]{\overline S h_j}_{\fbp[E]}\geq \intoo[3]{\sum_{k=1}^{K} \abs[0]{\overline S h_j(x_{j,k}^{*})}^p}^{\frac{1}{p}} = \intoo[3]{\sum_{k=1}^{K} \abs[0]{h_j\circ J_{E^*}(x_{j,k}^{*})}^p}^{\frac{1}{p}}\geq 1-\delta.
    \end{align*}
    Since $(h_j)$ is a $\delta$-net on $S_Y$, for every $h\in S_Y$ there is a $j$ such that $\norm{h-h_j}_{\fbp[E^{**}]}\leq \delta$, and thus
    \begin{align*}
        \norm[0]{\overline S h}_{\fbp[E]} & \geq \norm[0]{\overline S h_j}_{\fbp[E]}-\norm[0]{\overline S (h-h_j)}_{\fbp[E]}\\
        & \geq 1-\delta -\delta (1+\delta)\norm{S}\geq 1-2\delta
    \end{align*}
    for $\delta$ small enough. We conclude that
    \begin{equation*}
        \norm{Q}\leq \frac{(1+\delta)^2}{1-\delta},\,\,\norm[0]{Q^{-1}}\leq \frac{1}{1-2\delta}\,\,\text{and}\,\, \norm{Q}\norm[0]{Q^{-1}}\leq \frac{(1+\delta)^2}{(1-\delta)(1-2\delta)},
    \end{equation*}
    so choosing $\delta$ sufficiently small we can obtain $T=Q\circ R$ from $X$ onto $T(X)\subset \fbp[E]$ with distortion below $1+\epsilon$.
\end{proof}

Let us now take a closer look at the case $p=\infty$. Recall that $\FBLi[E]$ coincides with the lattice of positively homogeneous weak* continuous functions over $B_{E^*}$, denoted by $\Ccal_{ph}(B_{E^*})$ \cite[Proposition 2.2]{OTTT}. Thus, \Cref{theo: FBL of bidual embeds into bidual of FBL} in particular states that $\Ccal_{ph}(B_{E^{***}})$ embeds isometrically into $\Ccal_{ph}(B_{E^*})^{**}$.\\

A similar statement holds if we consider the whole space of weak* continuous functions instead of working only with the positively homogeneous ones, since for any Banach space $E$ the Banach lattice $\Ccal(B_{E^*})$ behaves as a free object over $E$ in the subcategory of $\Ccal(K)$-spaces in the following sense (cf. \cite[Theorem 5.4]{JLTTT}):\\ 

\begin{thm}\label{theo: free C(K) over E}
    Let $E$ be a Banach space. For every compact Hausdorff space $K$ and every bounded linear operator $T:E\rightarrow \Ccal(K)$, there exists a unique lattice homomorphism $\widetilde{T}: \Ccal(B_{E^*})\rightarrow \Ccal(K)$ such that $\widetilde{T}\circ \delta_E=T$ and $\widetilde{T}\uno_{E^*}=\norm{T}\uno_K$, where $\delta_E: E\rightarrow \Ccal(B_{E^*})$ denotes the isometric embedding that sends each $x\in E$ into the evaluation function $\delta_Ex(x^*)=x^*(x)$ for $x^*\in B_{E^*}$, and $\uno_{E^*}$ and $\uno_K$ denote the constant functions over $B_{E^*}$ and $K$, respectively. Moreover, $\norm[0]{\widetilde{T}}=\norm{T}$.
\end{thm}

In particular, given a Banach space $E$ and its canonical embedding $\delta_E$, we can consider the operator $\delta_E^{**}: E^{**}\rightarrow \Ccal(B_{E^*})^{**}$, which is an isometric embedding too. Note that $\Ccal(B_{E^*})^{**}$ is an $AM$-space with order unit $J_{\Ccal(B_{E^*})}\uno_{E^*}$. Moreover, the norm in $\Ccal(B_{E^*})^{**}$ coincides with the order unit norm induced by $J_{\Ccal(B_{E^*})}\uno_{E^*}$, so by Kakutani's Theorem (cf. \cite[Theorem 2.1.3]{MeyerNieberg}) $\Ccal(B_{E^*})^{**}$ is isometrically lattice isomorphic to some $\Ccal(K)$-space. Consequently, \Cref{theo: free C(K) over E} states that there exists a unique lattice homomorphism $\widetilde{\delta_E^{**}}: \Ccal(B_{E^{***}})\rightarrow \Ccal(B_{E^*})^{**}$ such that $\widetilde{\delta_E^{**}}\circ \delta_{E^{**}}=\delta_E^{**}$ and $\widetilde{\delta_E^{**}}\uno_{E^{***}}=J_{\Ccal(B_{E^*})}\uno_{E^*}$, which in addition has norm one. As it was mentioned below, we have a statement similar to \Cref{theo: FBL of bidual embeds into bidual of FBL} for this operator.

\begin{thm}\label{theo: continuous on ball of tridual embeds into bidual of continuous on ball of dual}
    Let $E$ be a Banach space. Then the operator $\widetilde{\delta_E^{**}}: \Ccal(B_{E^{***}})\rightarrow \Ccal(B_{E^*})^{**}$ is an isometric lattice embedding.
\end{thm}

\begin{proof}
    On the one hand, we notice that the composition operator 
    \begin{equation*}
        \fullfunction{C_{J_{E^*}}}{ \Ccal(B_{E^{***}})}{\ell_{\infty}(B_{E^*})}{h}{C_{J_{E^*}}h=h\circ J_{E^*}}
    \end{equation*}
    is an isometric lattice embedding as a consequence of Goldstine's Theorem (cf. \cite[Theorem 3.96]{FHHMZBanachSpaceTheory}).\\

    On the other hand, we observe that if we denote by $\eta_{x^*}$ the evaluation functional on $x^*\in B_{E^*}$ defined as
    \begin{equation*}
        \fullfunction{\eta_{x^*}}{ \Ccal(B_{E^{*}})}{\R}{f}{\eta_{x^*}(f)=f(x^*),}
    \end{equation*}
    then for every $h\in L:= \lat (\delta_{E^{**}}(E^{**})\cup \cbr{\uno_{E^{***}}})\subset \Ccal(B_{E^{***}})$ and $x^*\in B_{E^*}$ the following identity holds:
    \begin{equation*}
        \widetilde{\delta_E^{**}} h (\eta_{x^*})= C_{J_{E^*}}h (x^*).
    \end{equation*}
    Indeed, the equation is trivially true if $h=\uno_{E^{***}}$, and it follows easily when $h=\delta_{E^{**}}x^{**}$ for some $x^{**}\in E^{**}$ since $\delta_E^* \eta_{x^*}=x^*$. Moreover, $\eta_{x^*}$ is a lattice homomorphism from $\Ccal(B_E^*)$ to $\R$ and thus, an atom in $\Ccal(B_E^*)^*$, so $J_{\Ccal(B_E^*)^*}\eta_{x^*}$ is also a lattice homomorphism. Applying \cite[Exercise 8, p. 204]{AliprantisBurkinshaw} in order to represent every $h\in L$ as a finite lattice-linear expression depending only on some $(\delta_{E^{**}}x_i^{**})_{i=1}^m$ and $\uno_{E^{***}}$, and recalling that $C_{J_{E^*}}$ is a lattice homomorphism, we obtain the identity.\\

    Now, for every $h\in L$ it follows that
    \begin{align*}
         \norm{h}_{\Ccal(B_{E^{***}})} & = \norm[0]{C_{J_{E^*}}h}_{\ell_{\infty}(B_{E^*})} = \sup_{x^*\in B_{E^*}} \abs[0]{C_{J_{E^*}}h(x^*)} = \sup_{x^*\in B_{E^*}} \abs[0]{\widetilde{\delta_E^{**}} h (\eta_{x^*})}  \\
        &\leq \norm[0]{\widetilde{\delta_E^{**}} h}_{\Ccal(B_{E^*})^{**}} \leq \norm[0]{\widetilde{\delta_E^{**}}} \norm{h}_{\Ccal(B_{E^{***}})} = \norm{h}_{\Ccal(B_{E^{***}})}.
    \end{align*}
    Since the sublattice $L$ is dense in $\Ccal(B_{E^{***}})$ by Stone--Weierstra{\ss}'s Theorem (cf. \cite[Theorem 2.1]{MeyerNieberg}), we conclude that $$\norm{h}_{\Ccal(B_{E^{***}})}= \norm[0]{\widetilde{\delta_E^{**}} h}_{\Ccal(B_{E^*})^{**}}$$ for every $h\in \Ccal(B_{E^{***}})$.
\end{proof}

\section{Free $p$-convex dual Banach lattices}\label{sec:free dual BL}

The present section will be devoted to studying the existence of free objects in the category of duals of Banach lattices (and the subcategories of $p$-convex ones). Given $1\leq p\leq \infty$, we denote by $\mathcal{BL}^{*(p)}$ the category whose objects are duals of Banach lattices which are $p$-convex, and whose morphisms are weak*-to-weak* continuous (equivalently, adjoint) lattice homomorphisms.\\

\begin{rem}\label{rem: order continuos predual when p>1}
    Observe that when $p>1$ every object $X^*$ of the category $\mathcal{BL}^{*(p)}$ is canonically lattice complemented in its bidual $X^{***}$. Consider the contractive adjoint operator $J_X^*:X^{***}\rightarrow X^*$, which is a lattice quotient. Indeed, since $X^*$ is $p$-convex, $X$ must be $p^*$-concave, where $p^*=\frac{p}{p-1}<\infty$ (cf. \cite[Proposition 1.d.4]{LindenstraussTzafririVol2}). The $p^*$-concavity of $X$ for finite $p^*$ implies that $X$ is order continuous, and thus it is an ideal in $X^{**}$ (cf. \cite[Theorems 4.9 and 4.14]{AliprantisBurkinshaw}). In other words, the canonical embedding $J_X$ is an interval preserving isometric embedding, so $J_X^*$ is a contractive weak*-to-weak* continuous lattice quotient (cf. \cite[Theorem 1.4.19]{MeyerNieberg}). Finally, the composition with the canonical embedding $J_{X^*}$ yields a lattice projection.
\end{rem}

The above remark will be key in the proof of existence of a free object over a Banach space $E$ in the category $\mathcal{BL}^{*(p)}$ for $p>1$ in the following sense.

\begin{prop}\label{prop: bidual of FBLp is FBL*p}
    Let $E$ be a Banach space and $1<p\leq \infty$. Then, for every $p$-convex Banach lattice $X^*$ which is the dual of some Banach lattice and every linear and bounded operator $T: E\rightarrow X^*$ there exists a unique weak*-to-weak* continuous lattice homomorphism $\check{T}: \fbp[E]^{**} \rightarrow X^*$, such that $\check{T} \circ J_{\fbp[E]} \circ \phi_E = T$. Moreover, $\norm[0]{\check{T}}\leq M^{(p)}(X^*)\norm{T}$.
\end{prop}

\begin{proof}
    Since $X^*$ is $p$-convex, we can consider the following commutative diagram, where $\hat{T}$ is the unique extension of $T$ to $\fbp[E]$ as a lattice homomorphism:
    \begin{equation*}
	\xymatrix{
		\fbp[E]^{**} \ar@{->}[r]^{\,\,\,\,\,\,\,\,\,\,\,\,\hat{T}^{**}}& X^{***} \ar@{->}[d]^{J_X^*}  \\
		\fbp[E] \ar@{->}[r]^{\,\,\,\,\,\,\hat{T}} \ar@{->}[u]^{J_{\fbp[E]}} & X^* \\
        E \ar@{->}[ru]_{T} \ar@{->}[u]^{\phi_E} & 
	}
    \end{equation*}
    Since $J_X^*$ is a norm one, weak*-to-weak* continuous lattice homomorphism satisfying $J_X^*\circ J_{X^*}=I_{X^*}$, it follows that the operator $\check{T} := J_X^*\circ \hat{T}^{**}$ is also a weak*-to-weak* continuous lattice homomorphism such that $\check{T} \circ J_{\fbp[E]} \circ \phi_E = T$ and $\norm[0]{\check{T}}\leq M^{(p)}(X)\norm{T}$. Now, suppose there exists another weak*-to-weak* continuous lattice homomorphism $\check{S}: \fbp[E]^{**} \rightarrow X^*$ satisfying $\check{S} \circ J_{\fbp[E]} \circ \phi_E = T$. The uniqueness of extension of $T$ to $\fbp[E]$ implies that $\check{S} \circ J_{\fbp[E]} = \hat{T}$. Since $J_{\fbp[E]}$ has weak* dense range and both $\check{T}$ and $\check{S}$ are weak*-to-weak* continuous, they are uniquely determined by their values on $J_{\fbp[E]}(\fbp[E])$, which coincide. Hence, $\check{T} = \check{S}$.
\end{proof}

\begin{rem}\label{rem: argument fails for p=1}
    The previous argument cannot be applied to the case $p=1$ because there exists objects of $\mathcal{BL}^*$ that do not have an order continuous predual (see \cite{GLX} for a characterization of this property). Consider for instance $\mathcal{M}[0,1]$, the space of Borel measures over the interval $[0,1]$, which is the dual of the Banach lattice $\Ccal[0,1]$ but does not have any order continuous predual. Indeed, $\mathcal{M}[0,1]$ is an $AL$-space, so by \cite[Proposition 1.4.7]{MeyerNieberg}, all its predual Banach lattices must be $AM$-spaces. Since the only order continuous $AM$-spaces are of the form $c_0(\Gamma)$ for some set $\Gamma$ \cite[Lemma 2.7.12]{MeyerNieberg} and the dual of $c_0(\Gamma)$ is $\ell_1(\Gamma)$, which is never isomorphic to $\mathcal{M}[0,1]$ (note that $L_1[0,1]$ lattice embeds in $\mathcal{M}[0,1]$ but not in $\ell_1(\Gamma)$), we conclude that $\mathcal{M}[0,1]$ does not have an order continuous predual.
\end{rem}

However, we could still try to define a free object in the category $\mathcal{BL}^*$ of duals of Banach lattices in the following sense.

\begin{defn}\label{defi: free dual BL over a Bsp}
    Let $E$ be a Banach space. The \emph{free dual Banach lattice over $E$} is a Banach lattice $\fbl^*[E]$ which is the dual of some other Banach lattice $Z_E$, together with a linear isometric embedding $\iota_E: E\rightarrow \fbl^*[E]$, such that for every bounded linear operator $T$ from $E$ to the dual of some Banach lattice $X^*$, there exists a unique weak*-to-weak* continuous lattice homomorphism $\check{T}: \fbl^*[E] \rightarrow X^*$, such that $\check{T} \circ \iota_E = T$, and moreover, $\norm[0]{\check{T}}=\norm{T}$.
\end{defn}

By standard arguments it can be proved that whenever this object exists, it is unique. However, we do not know if such a free object exists for every Banach space $E$. The natural candidate to satisfy this definition is $\fbl[E]^{**}$. Nevertheless, we cannot hope for \Cref{prop: bidual of FBLp is FBL*p} to be true in general when $p=1$, as the following example shows.

\begin{prop}\label{prop: bidual of FBL is not FBL*E when E not oc predual}
    Let $X^*$ be the dual of a Banach lattice, so that $X^*$ does not admit an order continuous predual (take for instance $X^*=\mathcal M(0,1)$ as mentioned in \Cref{rem: argument fails for p=1}). The pair $(\fbl[X^*]^{**}, K_{X^*})$ does not have the universal property of \Cref{defi: free dual BL over a Bsp}, where $K_{X^*}=J_{\fbl[X^*]} \circ \phi_{X^*}$. More specifically, the identity operator $I_{X^*}$ on ${X^*}$ cannot be extended to a weak*-to-weak* continuous lattice homomorphism $\check{I}_{X^*}: \fbl[X^*]^{**} \rightarrow X^*$ such that $\check{I}_{X^*} \circ K_{X^*} = I_{X^*}$.
\end{prop}

\begin{proof}
    Suppose that there exists a weak*-to-weak* continuous lattice homomorphism $\check{I}_{X^*}: \fbl[X^*]^{**} \rightarrow X^*$ such that $\check{I}_{X^*} \circ K_{X^*} = I_{X^*}$. Let $S:X\rightarrow \fbl[X^*]^*$ be such that $S^*=\check{I}_{X^*}$. Note that $\check{I}_{X^*}$ is an onto lattice homomorphism, so $S$ must be an interval preserving embedding as a consequence of the Open Mapping Principle (cf. \cite[Theorem 2.25]{FHHMZBanachSpaceTheory}) and \cite[Theorem 1.4.19]{MeyerNieberg}.\\
    
    On the other hand, since $\check{I}_{X^*} \circ K_{X^*} = I_{X^*}$ and $\check{I}_{X^*} \circ J_{\fbl[X^*]}$ is a lattice homomorphism, the uniqueness of extension of $I_{X^*}$ to $\fbl[X^*]$ implies that $\check{I}_{X^*} \circ J_{\fbl[X^*]}= \hat{I}_{X^*}$. Considering the adjoint of this operator and composing it by the right with $J_X$, the canonical embedding of $X$ into its bidual, we get that 
    \begin{align*}
        \hat{I}_{X^*}^*\circ J_X & =  J_{\fbl[X^*]}^* \circ \check{I}_{X^*}^*\circ J_X =  J_{\fbl[X^*]}^* \circ S^{**}\circ J_X \\
        & = J_{\fbl[X^*]}^* \circ J_{\fbl[X^*]^*} \circ S = S,
    \end{align*}
    where in the last equality we have used that $J_{\fbl[X^*]}^* \circ J_{\fbl[X^*]^*}= I_{\fbl[X^*]^*}$. Using again the Open Mapping Principle (cf. \cite[Theorem 2.25]{FHHMZBanachSpaceTheory}) and \cite[Theorem 1.4.19]{MeyerNieberg} it is easy to check that $\hat{I}_{X^*}^*$ is an interval preserving embedding (in particular, it is injective). Thus, for every $x\in X_+$ it follows that
    \begin{equation*}
        \hat{I}_{X^*}^*(J_X[0,x])=S[0,x]=[0,Sx]=[0,\hat{I}_{X^*}^*(J_Xx)] = \hat{I}_{X^*}^*[0,J_Xx],
    \end{equation*}
    so $J_X[0,x]=[0,J_Xx]$ and $J_X$ is interval preserving. This is equivalent to $J_X(X)$ being an ideal of $X^{**}$, so by \cite[Theorem 2.4.2]{MeyerNieberg} $X$ must be order continuous, which is a contradiction.
\end{proof}

This shows that, in general, $\fbl[E]^{**}$ is not the free object over the Banach space $E$ in the sense of \Cref{defi: free dual BL over a Bsp}. However, we will provide a characterization of when $\fbl[E]^{**}=\fbl^*[E]$.

\begin{thm}\label{theo: bidual of FBL is FBL*E iff ell1 not complemented in E}
    Let $E$ be a Banach space. The following are equivalent:
    \begin{enumerate}
        \item $E$ does not contain a complemented subspace isomorphic to $\ell_1$.
        \item The pair $(\fbl[E]^{**}, K_E)$ with $K_E=J_{\fbl[E]} \circ \phi_E$ coincides with the free dual Banach lattice over $E$, $(\fbl^*[E],\iota_E)$.
    \end{enumerate}
\end{thm}

\begin{proof}
    $(1)\Rightarrow (2)$ Let $X$ be a Banach lattice and $T:E \rightarrow X^*$ a bounded linear operator. By \cite[Theorem 4]{BessagaPelczynski}, $E$ does not contain a complemented copy of $\ell_1$ if and only if $E^*$ does not contain a copy of $c_0$. Thus, applying \cite[Theorem I.2]{GhoussoubJohnson} we obtain that the operator $T^*\circ J_X: X\rightarrow E^*$ can be factored as $T^*\circ J_X=S\circ R$, where $R:X\rightarrow Z$, $S:Z\rightarrow E^*$ and $Z$ is an order continuous Banach lattice. Moreover, $Z$ is the norm completion of a quotient of $X$ by an ideal, and $R$ is the canonical quotient, so it is a lattice homomorphism, and, by \cite[Proposition II.2.5]{Schaefer}, it is almost interval preserving. Note that, in particular, $T=R^*\circ S^*\circ J_E$, and we can consider the following commutative diagram:
    \begin{equation*}
	\xymatrix{
		\fbl[E]^{**} \ar@{->}[r]^{\,\,\,\,\,\,\,\,\,\,\,\,\widehat{S^*\circ J_E}^{**}}& Z^{***} \ar@{->}[d]^{J_Z^*} &  \\
		\fbl[E] \ar@{->}[r]^{\,\,\,\,\,\,\widehat{S^*\circ J_E}} \ar@{->}[u]^{J_{\fbl[E]}} & Z^* \ar@{->}[r]^{R^*} & X^* \\
        E \ar@{->}[rru]_{T} \ar@{->}[ru]^{S^*\circ J_E} \ar@{->}[u]^{\phi_E} &  & 
	}
    \end{equation*}
    Since $Z$ is order continuous, by \cite[Theorem 2.4.2]{MeyerNieberg} $J_Z$ must be interval preserving, so using \cite[Theorem 1.4.19]{MeyerNieberg} we conclude that the operator $\check{T}= R^*\circ J_Z^* \circ (\widehat{S^*\circ J_E})^{**}$ is a weak*-to-weak* continuous lattice homomorphism such that $\check{T} \circ J_{\fbl[E]} \circ \phi_E = T$. That this extension is unique follows from the universal property of $\fbl[E]$ and the weak* density of $\fbl[E]$ inside its bidual. Since $\norm{R}\leq \norm{T}$ and $\norm{S}\leq 1$ by construction (see the proof in \cite[Theorem I.2]{GhoussoubJohnson} for the details), it is straightforward to check that $\norm[0]{\check{T}}=\norm{T}$. Thus, we have proved that the pair $(\fbl[E]^{**}, J_{\fbl[E]} \circ \phi_E)$ satisfies the conditions of \Cref{defi: free dual BL over a Bsp}. \\

    $(2)\Rightarrow (1)$ Let us suppose the pair $(\fbl[E]^{**}, K_E)$ has the universal property of \Cref{defi: free dual BL over a Bsp} but $E$ contains a complemented copy of $\ell_1$. In particular, there exists a quotient $P:E\rightarrow \ell_1$. Let us consider the separable space $L_1[0,1]$, which is a quotient of $\ell_1$ through a certain $Q:\ell_1\rightarrow L_1[0,1]$, and the operator $R:L_1[0,1]\rightarrow \mathcal M[0,1]$ given for each Borel set $A\subset[0,1]$ and $g\in L_1[0,1]$ by
    \begin{equation*}
        Rg (A)=\int_A gd\lambda,
    \end{equation*}
    where $\lambda$ denotes the Lebesgue measure over $[0,1]$ and $\mathcal{M}[0,1]$ represents the space of Borel measures over the interval $[0,1]$, which is the dual of $\Ccal[0,1]$. Note that $R(B_{L_1[0,1]})$ is weak* dense in $B_{\mathcal{M}[0,1]}$. Indeed, given $t\in [0,1]$, let $(I_{t,n})_n$ be a sequence of intervals of $[0,1]$ of length $\frac{1}{n}$ decreasing to $\cbr{t}$. Then, it is straightforward to check that the measures $\mu_{t,n}=\lambda(I_{t,n})^{-1}R\chi_{I_{t,n}}\in R(B_{L_1[0,1]})$ converge weak* to $\eta_t$, the Dirac measure on $t$. Thus, $\cbr{\pm\eta_t:t\in [0,1]}$, which is the set of extreme points of $B_{\mathcal{M}[0,1]}$ \cite[Proposition 2.1.2]{MeyerNieberg}, is contained in the weak* closure of $R(B_{L_1[0,1]})$. By Krein-Milman Theorem, we conclude that the whole $B_{\mathcal{M}[0,1]}$ lies in the weak* closure of $R(B_{L_1[0,1]})$.\\

    Now, consider the operator $T=R\circ Q\circ P:E\rightarrow \mathcal{M}[0,1]$. By our assumption, there exists a weak*-to-weak* continuous lattice homomorphism $\check{T}:\fbl[E]^{**}\rightarrow \mathcal{M}[0,1]$ such that $\check{T}\circ K_E=T$. We claim that $\check{T}$ is onto. Indeed, fix some $\mu\in B_{\mathcal{M}[0,1]}$. We have seen above that $R(B_{L_1[0,1]})$ is weak* dense in $B_{\mathcal{M}[0,1]}$, so there exists some net $(g_{\alpha})_{\alpha}\subset B_{L_1[0,1]}$ such that $(Rg_{\alpha})_{\alpha}$ converges weak* to $\mu$. Since $Q\circ P$ is onto, by the Open Mapping Theorem \cite[Corollary 2.26]{FHHMZBanachSpaceTheory} there exists some $r>0$ such that $B_{L_1[0,1]}\subset Q\circ P(rB_E)$. Hence, we can find a bounded net $(x_{\alpha})_{\alpha}\subset E$ such that $Q\circ P x_{\alpha}=g_{\alpha}$. Now, using the weak* compactness of $B_{\fbl[E]^{**}}$, we can extract a subnet $(x_{\beta})_{\beta}\subset (x_{\alpha})_{\alpha}$ such that $(K_Ex_{\beta})_{\beta}$ is weak* convergent to some $f^{**}\in \fbl[E]^{**}$. Evaluating on $\check{T}$ we obtain that $\check{T}\circ K_E x_{\beta}=Rg_{\beta}$ converges weak* to both $\mu$ and $\check{T}f^{**}$, so $\mu=\check{T}f^{**}$ and $\check{T}$ is an onto lattice homomorphism. By \cite[Proposition II.2.5]{Schaefer}, this implies that $\check{T}$ is also interval preserving. Let $S:\Ccal[0,1]\rightarrow \fbl[E]^*$ be an operator such that $\check{T}=S^*$. By the Open Mapping Principle (cf. \cite[Theorem 2.25]{FHHMZBanachSpaceTheory}) and \cite[Theorem 1.4.19]{MeyerNieberg} it follows that $S$ must be an interval preserving embedding of $\Ccal[0,1]$ into $\fbl[E]^*$.\\ 
    
    Note that $\fbl[E]^*$ is a dual of a Banach lattice, so in particular it is ($\sigma$-)Dedekind complete. Clearly, $\Ccal[0,1]$ is not $\sigma$-Dedekind complete. Since any ideal of a $\sigma$-Dedekind complete Banach lattice is again $\sigma$-Dedekind complete, it follows that $\Ccal[0,1]$ cannot be embedded as an ideal in $\fbl[E]^*$. Thus, $E$ cannot contain a complemented copy of $\ell_1$.
\end{proof}

It remains an open question whether $\fbl^*[E]$ may exist when $E$ contains a complemented subspace isomorphic to $\ell_1$. The following result provides some properties that $\fbl^*[E]$ must satisfy, in case it exists. Note that all of them become trivial when $\fbl^*[E]$ happens to coincide with $\fbl[E]^{**}$, that is, when $E$ does not contain a complemented copy of $\ell_1$.

\begin{lem}\label{lemm: known facts about FBL*[E]}
    Let $E$ be a Banach space, and assume that $\fbl^*[E]$ exists. Then:
    \begin{enumerate}
        \item The unique extension $\hat{\iota}_E: \fbl[E]\rightarrow \fbl^*[E]$ of $\iota_E$ as a lattice homomorphism is an isometric embedding.
        \item $\check{K}_E: \fbl^*[E]\rightarrow \fbl[E]^{**}$ is an onto interval preserving lattice homomorphism, where $K_E= J_{\fbl[E]}\circ \phi_E$.
        \item $\overline{\hat{\iota}_E(\fbl[E])}^{w^*}$ is a subspace of $\fbl^*[E]$ linearly isometric to $\fbl[E]^{**}$, and it is complemented in $\fbl^*[E]$ by means of a positive contractive weak*-to-weak* continuous projection given by $J_{Z_E}^*\circ \hat{\iota}_E^{**}\circ\check{K}_E$.
        \item $\fbl[E]^*$ embeds isometrically in $Z_E$ as an ideal, and it is complemented in $Z_E$ by means of a positive contractive projection.
    \end{enumerate}
\end{lem}

\begin{proof}
    $(1)$ Note that
    \begin{equation*}
        \check{K}_E\circ \hat{\iota}_E\circ \phi_E= \check{K}_E\circ \iota_E = K_E= J_{\fbl[E]}\circ \phi_E.
    \end{equation*}
    Since both $\check{K}_E\circ \hat{\iota}_E$ and $J_{\fbl[E]}$ are lattice homomorphisms, the uniqueness of extension of operators to $\fbl[E]$ implies that $\check{K}_E\circ \hat{\iota}_E = J_{\fbl[E]}$. Now, using that $J_{\fbl[E]}$ is an isometric embedding and both $\hat{\iota}_E$ and $\check{K}_E$ have norm one, we obtain that
    \begin{equation*}
        \norm{f}_{\fbl[E]}= \norm[0]{J_{\fbl[E]}f}_{\fbl[E]^{**}} \leq \norm{\hat{\iota}_E f}_{\fbl^*[E]} \leq \norm{f}_{\fbl[E]}
    \end{equation*}
    for every $f\in \fbl[E]$.\\

    $(2)$ First, let us show that $\check{K}_E$ is onto. Given some $f^{**}\in B_{\fbl[E]^{**}}$, by Goldstine's Theorem we know that there exists a net $(f_{\alpha})_{\alpha}\subset B_{\fbl[E]}$ such that $(J_{\fbl[E]} f_{\alpha})_{\alpha}$ is weak* convergent to $f^{**}$ in $\fbl[E]^{**}$. Now, observe that $(\hat{\iota}_E f_{\alpha})_{\alpha}\subset B_{\fbl^*[E]}$, which is a weak* compact set, so there exists a subnet $(\hat{\iota}_E f_{\beta})_{\beta}$ weak* convergent to some $z^*\in B_{\fbl^*[E]}$. Using that $\check{K}_E$ is weak*-to-weak* continuous and satisfies $\check{K}_E\circ \hat{\iota}_E = J_{\fbl[E]}$, we conclude that $f^{**}=\check{K}_E z^*$. Thus, $\check{K}_E$ is an onto lattice homomorphism, and by \cite[Proposition II.2.5]{Schaefer}, it follows that it is interval preserving.\\

    $(3)$ Consider the operator $J_{Z_E}^*\circ \hat{\iota}_E^{**}: \fbl[E]^{**}\rightarrow \fbl^*[E]$. Using that $J_{Z_E}^*\circ J_{Z_E^*}$ is the identity on $Z_E^*=\fbl^*[E]$, it follows that
    \begin{equation*}
        \check{K}_E\circ J_{Z_E}^*\circ \hat{\iota}_E^{**}\circ J_{\fbl[E]}= \check{K}_E\circ J_{Z_E}^*\circ J_{Z_E^*}\circ \hat{\iota}_E = \check{K}_E\circ \hat{\iota}_E = J_{\fbl[E]}.
    \end{equation*}
    Since $\check{K}_E\circ J_{Z_E}^*\circ \hat{\iota}_E^{**}$ is  weak*-to-weak* continuous and $J_{\fbl[E]}$ has weak* dense range, we conclude that $\check{K}_E\circ J_{Z_E}^*\circ \hat{\iota}_E^{**}=I_{\fbl[E]^{**}}$. This implies that $J_{Z_E}^*\circ \hat{\iota}_E^{**}$ is an isometric embedding. Indeed, given any $f^{**}\in \fbl[E]^{**}$, we have
    \begin{align*}
        \norm{f^{**}}_{\fbl[E]^{**}} & = \norm[0]{\check{K}_E\circ J_{Z_E}^*\circ \hat{\iota}_E^{**}f^{**}}_{\fbl[E]^{**}}\\
        & \leq  \norm[0]{J_{Z_E}^*\circ \hat{\iota}_E^{**}f^{**}}_{\fbl^*[E]} \leq \norm{f^{**}}_{\fbl[E]^{**}}.
    \end{align*}
    The next step is to show that the range of $J_{Z_E}^*\circ \hat{\iota}_E^{**}$ is the subspace $U_E=\overline{\hat{\iota}_E(\fbl[E])}^{w^*}$. The inclusion $J_{Z_E}^*\circ \hat{\iota}_E^{**}(\fbl[E]^{**})\subset U_E$ is straightforward, since $J_{Z_E}^*\circ \hat{\iota}_E^{**}$ is  weak*-to-weak* continuous and $J_{\fbl[E]}$ has weak* dense range. To check the reverse inclusion, let $z^*\in U_E$ and $(f_{\alpha})_{\alpha}\subset \fbl[E]$ such that $(\hat{\iota}_E f_{\alpha})_{\alpha}$ is weak* convergent to $z^*$. Using the weak* continuity of $\check{K}_E$ and the identity  $\check{K}_E\circ \hat{\iota}_E = J_{\fbl[E]}$ it follows that $( J_{\fbl[E]} f_{\alpha})_{\alpha}$ is weak* convergent to $\check{K}_E z^*$. Now, recall that $J_{Z_E}^*\circ \hat{\iota}_E^{**}\circ J_{\fbl[E]} = \hat{\iota}_E$, so $(\hat{\iota}_E f_{\alpha})_{\alpha}$ weak* converges to $J_{Z_E}^*\circ \hat{\iota}_E^{**}\circ \check{K}_E z^*$. The uniqueness of the weak* limit implies that $J_{Z_E}^*\circ \hat{\iota}_E^{**}\circ \check{K}_E z^*= z^*$, so $U_E\subset J_{Z_E}^*\circ \hat{\iota}_E^{**}(\fbl[E]^{**})$. We conclude that $J_{Z_E}^*\circ \hat{\iota}_E^{**}\circ \check{K}_E$ is a positive contractive weak*-to-weak* continuous projection over $U_E$. \\

    $(4)$ By definition $\check{K}_E$ is weak*-to-weak* continuous, so it is the adjoint operator of some $M_E: \fbl[E]^*\rightarrow Z_E$. Since $\check{K}_E$ is an onto interval preserving lattice homomorphism, by the Open Mapping Principle (cf. \cite[Theorem 2.25]{FHHMZBanachSpaceTheory}) and \cite[Theorem 1.4.19]{MeyerNieberg} it follows that $M_E$ is an interval preserving lattice embedding. From the proof of $(3)$ we know that
    \begin{equation*}
        (\hat{\iota}^*_E\circ J_{Z_E}\circ M_E)^*=\check{K}_E\circ J_{Z_E}^*\circ \hat{\iota}_E^{**}=I_{\fbl[E]^{**}},
    \end{equation*}
    so the operator $\hat{\iota}^*_E\circ J_{Z_E}\circ M_E$ must be the identity on $\fbl[E]^*$. Proceeding as before, we can see that $M_E$ is an isometry, and that $M_E \circ \hat{\iota}^*_E\circ J_{Z_E}$ is a positive contractive projection.
\end{proof}

In addition to the previous result, we present a compilation of statements equivalent to $\fbl[E]^{**}$ being isomorphic to $\fbl^*[E]$ in the category $\mathcal{BL}^*$. 

\begin{prop}\label{prop: equivalences of FBL*E=FBLE**}
    Let $E$ be a Banach space, and assume that $\fbl^*[E]$ exists. The following are equivalent:
    \begin{enumerate}
        \item The pair $(\fbl[E]^{**}, K_E)$ coincides with the free dual Banach lattice over $E$, $(\fbl^*[E],\iota_E)$.
        \item $\hat{\iota}_E: \fbl[E]\rightarrow \fbl^*[E]$ has weak* dense range.
        \item $\overline{\hat{\iota}_E(\fbl[E])}^{w^*}$ is a sublattice of $\fbl^*[E]$.
        \item $J_{Z_E}^*\circ \hat{\iota}_E^{**}: \fbl[E]^{**}\rightarrow \fbl^*[E]$ is a lattice homomorphism.
        \item $\check{K}_E: \fbl^*[E]\rightarrow \fbl[E]^{**}$ is a lattice and weak*-to-weak* continuous isometric isomorphism.
        \item There exists a Banach lattice $Y$ such that $\fbl^*[E]=Y^{**}$ and $\iota_E(E)\subset J_Y(Y)$.
        \item The unique extension $\check{j}_E:\fbl^*[E]\rightarrow \fbl^*[E]^{**}$ of $j_E=J_{\fbl^*[E]}\circ \iota_E$ as a weak*-to-weak* continuous lattice homomorphism is an isometric embedding.
    \end{enumerate}
\end{prop}

\begin{proof}
    $(1)\Rightarrow (2)$ Let $\widetilde{\iota}_E:\fbl[E]^{**}\rightarrow \fbl^*[E]$ be the unique extension of $\iota_E$ to $\fbl[E]^{**}$ as a weak*-to-weak* continuous lattice homomorphism. By standard arguments we can show that $\check{K}_E$ and $\widetilde{\iota}_E$ are mutually inverse. In particular, $\widetilde{\iota}_E$ is onto. Moreover, $\hat{\iota}_E=\widetilde{\iota}_E\circ J_{\fbl[E]}$, $J_{\fbl[E]}$ has weak* dense range and  $\widetilde{\iota}_E$ is weak*-to-weak* continuous, so $\hat{\iota}_E$ has weak* dense range. \\

    $(2)\Rightarrow (3)$ Trivial, since $\overline{\hat{\iota}_E(\fbl[E])}^{w^*}$ coincides with $\fbl^*[E]$.\\

    $(3)\Rightarrow (4)$ Recall from \Cref{lemm: known facts about FBL*[E]} that $J_{Z_E}^*\circ \hat{\iota}_E^{**}(\fbl[E]^{**}) =\overline{\hat{\iota}_E(\fbl[E])}^{w^*}$, and denote this subspace by $U_E$. Since it is a sublattice, for any $f_1^{**}\in \fbl[E]^{**}$ there must exist some $f_0^{**}\in \fbl[E]^{**}$ such that
    \begin{equation*}
        \abs[0]{J_{Z_E}^*\circ \hat{\iota}_E^{**} f_1^{**}} = J_{Z_E}^*\circ \hat{\iota}_E^{**} f_0^{**}.
    \end{equation*}
    Applying the lattice homomorphism $\check{K}_E$ and using that $\check{K}_E\circ J_{Z_E}^*\circ \hat{\iota}_E^{**} =I_{\fbl[E]^{**}}$ we obtain that $f_0^{**} = \abs[0]{f_1^{**}}$. Thus, 
    \begin{equation*}
        \abs[0]{J_{Z_E}^*\circ \hat{\iota}_E^{**} f_1^{**}}  = J_{Z_E}^*\circ \hat{\iota}_E^{**} \abs[0]{f_1^{**}},
    \end{equation*}
    so $J_{Z_E}^*\circ \hat{\iota}_E^{**}$ must be a lattice homomorphism.\\

    $(4)\Rightarrow (5)$ On the one hand we know from \Cref{lemm: known facts about FBL*[E]} that $\check{K}_E\circ J_{Z_E}^*\circ \hat{\iota}_E^{**}=I_{\fbl[E]^{**}}$. On the other hand, it follows from the hypothesis that $J_{Z_E}^*\circ \hat{\iota}_E^{**}\circ\check{K}_E$ is a weak*-to-weak* continuous lattice homomorphism such that $J_{Z_E}^*\circ \hat{\iota}_E^{**}\circ\check{K}_E\circ \iota_E=\iota_E$. The uniqueness of extension to $\fbl^*[E]$ yields that $J_{Z_E}^*\circ \hat{\iota}_E^{**}\circ\check{K}_E=I_{\fbl^*[E]}$, so $\check{K}_E$ is a lattice isomorphism. Since both $\check{K}_E$ and $J_{Z_E}^*\circ \hat{\iota}_E^{**}$ have norm one, it is also an isometry.\\

    $(5)\Rightarrow (6)$ Just take $Y=\fbl[E]$.\\

    $(6)\Rightarrow (7)$ Let us show that $\check{j}_E=J_Y^{**}$, which is an isometric embedding. Indeed, the condition $\iota_E(E)\subset J_Y(Y)$ implies that the operator $J_Y^{-1}\circ \iota_E:E\rightarrow Y$ is well-defined. Thus,
    \begin{equation*}
        J_Y^{**}\circ \iota_E= J_Y^{**}\circ J_Y\circ J_Y^{-1} \circ \iota_E = J_{Y^{**}}\circ J_Y\circ J_Y^{-1} \circ \iota_E = J_{\fbl^*[E]} \circ \iota_E =j_E.
    \end{equation*}
    The uniqueness of extension of $j_E$ to $\fbl^*[E]$ implies that $\check{j}_E=J_Y^{**}$, as we wanted.\\

    $(7)\Rightarrow (5)$ We know from \Cref{lemm: known facts about FBL*[E]} that $\check{K}_E$ is onto. Let us see that it is also an isometric embedding, and hence an isomorphism. To do so, we will check that $\hat{\iota}_E^{**}\circ\check{K}_E= \check{j}_E$. Indeed, $\hat{\iota}_E^{**}\circ\check{K}_E$ is a weak*-to-weak* continuous lattice homomorphism satisfying
    \begin{equation*}
        \hat{\iota}_E^{**}\circ\check{K}_E \circ \iota_E=\hat{\iota}_E^{**}\circ  J_{\fbl[E]}\circ \phi_E = J_{\fbl^*[E]}\circ \hat{\iota}_E\circ  \phi_E =  J_{\fbl^*[E]}\circ \iota_E = j_E,
    \end{equation*}
    so the uniqueness of extension of $j_E$ implies that  $\hat{\iota}_E^{**}\circ\check{K}_E= \check{j}_E$. Since both $\hat{\iota}_E^{**}$ and $\check{j}_E$ are isometric embeddings, $\check{K}_E$ is too.\\

    $(5)\Rightarrow (1)$ Trivial, since $\check{K}_E$ is an isomorphism in the category $\mathcal{BL}^*$.\\
\end{proof}

\Cref{prop: equivalences of FBL*E=FBLE**} provides some useful information in the search for an answer to the question of the existence of $\fbl^*[E]$ when $E$ contains a complemented copy of $\ell_1$. For instance, we can affirm that if $\fbl^*[E]$ exists, then it cannot be the bidual of a Banach lattice and cannot contain $\fbl[E]^{**}$ as a sublattice (at least in the canonical way described in \Cref{lemm: known facts about FBL*[E]}).\\

\section{Free duals of $AL$-spaces}\label{sec:AM-spaces}

The results from the previous section can be adapted to the context of $AM$-spaces with order unit (i.e., $\Ccal(K)$-spaces). More specifically, we can define the category $\mathcal{AL}^{*}$, whose objects are $AM$-spaces which are duals of Banach lattices (equivalently, duals of $AL$-spaces \cite[Proposition 1.4.7]{MeyerNieberg}) and whose morphisms are weak*-to-weak* continuous lattice homomorphisms that preserve the order unit, meaning that normalized morphisms send the order unit of the domain to the order unit of the image. Throughout this section, the order unit of an $AM$-space $X$ is the unique element $e\in X_+$ such that $[-e,e]=B_{X}$. \\ 

Note that by \cite[Proposition 1.4.7]{MeyerNieberg} every $AM$-space $X^*$ which is the dual of some Banach lattice $X$ has an order unit $e$, and the norm in $X^*$ coincides with the order unit norm $\norm{\cdot}_e$ (equivalently, $[-e,e]=B_{X^*}$), so the restriction imposed on the morphisms of the category $\mathcal{AL}^{*}$ is consistent. Moreover, by Kakutani's and Dixmier's Theorems (cf. \cite[Theorems 2.1.3 and 2.1.7]{MeyerNieberg} an object is in $\mathcal{AL}^{*}$ if and only if it is isometrically lattice isomorphic to a $\Ccal(K)$ space with $K$ hyper-stonean.\\

As $AM$-spaces are in particular $\infty$-convex, a property similar to \Cref{rem: order continuos predual when p>1} holds for the category $\mathcal{AL}^{*}$, since $J_X^*$ preserves the order unit. This is a consequence of the equality $J_X^*\circ J_{X^*}=I_{X^*}$, together with the simple fact that if $e$ is the order unit of the $AM$-space $X^*$, then $J_{X^*}e$ is the order unit of $X^{***}$.\\

Thus, we can formulate a statement analogous to \Cref{prop: bidual of FBLp is FBL*p} in this new setting. The proof is essentially the same, so we will skip some of the details.

\begin{prop}\label{prop: bidual of C(BE*) is FBL*infty1}
    Let $E$ be a Banach space. Then, for every $AM$-space $X^*$ which is the dual of some Banach lattice and every linear and bounded operator $T: E\rightarrow X^*$ there exists a unique weak*-to-weak* continuous lattice homomorphism $\check{T}: \Ccal(B_{E^*})^{**} \rightarrow X^*$, such that $\check{T} \circ J_{\Ccal(B_{E^*})} \circ \delta_E = T$ and $\check{T} \circ J_{\Ccal(B_{E^*})} \uno_{E^*}= \norm{T}e_{X^*}$, where $\uno_{E^*}$ is the constant function on $B_{E^*}$ and $e_{X^*}$ is the order unit of $X^*$. Moreover, $\norm[0]{\check{T}}=\norm{T}$.
\end{prop}

\begin{proof}
    First, note that clearly $\Ccal(B_{E^*})^{**}$ is an object of $\mathcal{AL}^{*}$. Now, let us consider the following commutative diagram, where $\widetilde{T}$ is the extension of $T$ to $\Ccal(B_{E^*})$ given by \Cref{theo: free C(K) over E}:
    \begin{equation*}
	\xymatrix{
		\Ccal(B_{E^*})^{**} \ar@{->}[r]^{\,\,\,\,\,\,\,\,\,\widetilde{T}^{**}}& X^{***} \ar@{->}[d]^{J_X^*}  \\
		\Ccal(B_{E^*}) \ar@{->}[r]^{\,\,\,\,\,\,\widetilde{T}} \ar@{->}[u]^{J_{\Ccal(B_{E^*})}} & X^* \\
        E \ar@{->}[ru]_{T} \ar@{->}[u]^{\delta_E} & 
	}
    \end{equation*}
    The rest of the proof is similar to the one in \Cref{prop: bidual of FBLp is FBL*p}.
\end{proof}

To finish this section, let us focus on the relation between $\Ccal(B_{E^*})$ and $\FBLi[E]$, which can be represented as $\Ccal_{ph}(B_{E^*})$ in a canonical way. These Banach lattices are never lattice isomorphic. Indeed, $\Ccal(B_{E^*})$ always has an order unit, while in $\FBLi[E]$ this only happens when $E$ has finite dimension \cite[Proposition 9.1]{OTTT}. In that case, $\FBLi[E]$ is isometrically lattice isomorphic to $\Ccal(S_{E^*})$, which by Banach--Stone's Theorem can never be lattice isomorphic to $\Ccal(B_{E^*})$ (cf. \cite[Corollary 3.2.11]{MeyerNieberg}). Neither can $\FBLi[E]$ be complemented in $\Ccal(B_{E^*})$ by means of a positive projection when $E$ has infinite dimension: if there was such a positive projection $P: \Ccal(B_{E^*}) \rightarrow \Ccal(B_{E^*})$ over $\FBLi[E]\equiv \Ccal_{ph}(B_{E^*})$, it would follow that $P\uno_{E^*}$ should be an order unit of $\FBLi[E]$, which by \cite[Proposition 9.1]{OTTT} is not possible.\\

This last situation no longer holds when we compare the bidual spaces instead: $\FBLi[E]^{**}$ is complemented in $\Ccal(B_{E^*})^{**}$ with a positive and contractive projection. This is a consequence of the fact that every Dedekind complete $AM$-space is injective (cf. \cite[Theorem 3.2.4]{MeyerNieberg}), since every dual of a Banach lattice is Dedekind complete. Nevertheless, in this particular situation we have a canonical construction of the projection, and it is a lattice homomorphism.

\begin{prop}\label{prop: bidual of FBLinfty complemented in bidual of C(BE*)}
    Let $K_E=J_{\FBLi[E]}\circ \phi_E$ denote the canonical embedding of $E$ into $\FBLi[E]^{**}$ and $\check{K}_E: \Ccal(B_{E^*})^{**}\rightarrow \FBLi[E]^{**}$ the extension of $K_E$ given by \Cref{prop: bidual of C(BE*) is FBL*infty1}. Then the operator $\hat{\delta}^{**}_E \circ \check{K}_E: \Ccal(B_{E^*})^{**}\rightarrow \Ccal(B_{E^*})^{**}$ is a contractive lattice projection over $\FBLi[E]^{**}$, where $\hat{\delta}_E$ is the canonical embedding of $\FBLi[E]$ into $\Ccal(B_{E^*})$ given by the restriction to $B_{E^*}$.
\end{prop}

\begin{proof}
    It is clear that $\hat{\delta}_E^{**}: \FBLi[E]^{**} \rightarrow \Ccal(B_{E^*})^{**}$ is an isometric lattice embedding. From the following commutative diagram,
    \begin{equation*}
	\xymatrix{
		\FBLi[E]^{**} \ar@{->}[r]^{\,\,\,\,\,\,\hat{\delta}_E^{**}} & \Ccal(B_{E^*})^{**} \ar@{->}[r]^{\check{K}_E} & \FBLi[E]^{**}  \\
		\FBLi[E] \ar@{->}[u]^{J_{\FBLi[E]}} \ar@{->}[r]^{\,\,\,\,\,\,\hat{\delta}_E} & \Ccal(B_{E^*})  \ar@{->}[u]^{J_{\Ccal(B_{E^*})}} & \FBLi[E] \ar@{->}[u]_{J_{\FBLi[E]}}  \\
        &  E \ar@{->}[lu]^{\phi_E} \ar@{->}[u]^{\delta_E} \ar@{->}[ru]_{\phi_E} &
	}
    \end{equation*}
    it is clear that
    \begin{equation*}
        \check{K}_E\circ \hat{\delta}_E^{**} \circ K_E= \check{K}_E\circ J_{\Ccal(B_{E^*})}\circ \hat{\delta}_E \circ \phi_E = \check{K}_E\circ J_{\Ccal(B_{E^*})}\circ \delta_E = K_E.
    \end{equation*}
    By the uniqueness of extension of $K_E$ to $\FBLi[E]^{**}$ established in \Cref{prop: bidual of FBLp is FBL*p}, we conclude that $\check{K}_E\circ \hat{\delta}_E^{**}$ must be the identity on $\FBLi[E]^{**}$, so the result follows.
\end{proof}

\begin{rem}\label{rema: bidual of FBLinfty and bidual of C(BE*) not canonically isomorphic}
    Note that $\hat{\delta}_E^{**}$ is an embedding, but never surjective. Indeed, if $\hat{\delta}_E^{**}$ were a surjective isomorphism, since it is also a lattice homomorphism, it would preserve the order unit. Let us denote by $e$ the order unit of $\FBLi[E]^{**}$. Then, $\hat{\delta}_E^{**} e$ is also an order unit in $\Ccal(B_{E^*})^{**}$, so there exists some $\lambda \geq 0$ such that $J_{\Ccal(B_{E^*})}\uno_{E^*}\leq \lambda \hat{\delta}_E^{**} e$. In particular, evaluating both elements in the evaluation functional $\eta_0\in \Ccal(B_{E^*})^*$ we obtain that
    \begin{align*}
        1 & = \uno_{E^*}(0)= \eta_0(\uno_{E^*})= J_{\Ccal(B_{E^*})}\uno_{E^*}(\eta_0)\\
        & \leq \lambda (\hat{\delta}_E^{**} e) (\eta_0)= \lambda e (\hat{\delta}_E^*\eta_0) = \lambda e (\widehat{0})= \lambda e (0)=0,
    \end{align*}
    where we have used that $\hat{\delta}_E^*\eta_x^*=\widehat{x^*}$ for every $x^*\in B_{E^*}$ and that $\widehat{0}$, the extension of the null functional $0\in E^*$ to $\FBLi[E]$ as a lattice homomorphism, is the null functional $0\in \FBLi[E]^*$. This leads us to a contradiction, so $\hat{\delta}_E^{**}$ cannot be an isomorphism.
\end{rem}

\section*{Acknowledgements}

We wish to thank Antonio Avil\'es, David de Hevia, Mitchell A. Taylor and Timur Oikhberg for interesting discussions related to the content of this paper. We are also grateful to the anonymous referees for their comments and suggestions.


\begin{thebibliography}{99}

\bibitem{AliprantisBurkinshaw}
C. D. Aliprantis, O. Burkinshaw, \textit{Positive operators}. Springer, Dordrecht, 2006.

\bibitem{AMR2} 
A.~Avil\'es, G.~Mart\'{i}nez-Cervantes, and J.D.~Rodr\'iguez-Abell\'an, \textit{On projective Banach lattices of the form $C(K)$ and $\fbl[E]$}.  J. Math. Anal. Appl. \textbf{489} (2020), no. 1, 124129, 11 pp.

\bibitem{AMRT}
A.~Avil\'es, G.~Mart\'{i}nez-Cervantes, J.~Rodr\'{\i}guez, and P.~Tradacete, \textit{A Godefroy-Kalton principle for free Banach lattices}. Israel J. Math. \textbf{247} (2022), 433--458.

\bibitem{ART}
A.~Avil\'es, J.~Rodr\'{\i}guez, and P.~Tradacete, \textit{The free Banach lattice generated by a Banach space}. J. Funct. Anal. \textbf{274} (2018), no. 10, 2955--2977.

\bibitem{BessagaPelczynski}
C. Bessaga and A. Pe\l czy\'{n}ski, \textit{On bases and unconditional convergence of series in Banach spaces}. Studia Math. \textbf{17} (1958), 151--164.
  
\bibitem{DMRR2}
S.~Dantas, G.~Mart\'inez-Cervantes, J.~Rodr\'iguez Abell\'an, and A.~Rueda Zoca, \textit{Norm-attaining lattice homomorphisms}. Rev. Mat. Iberoam. \textbf{38} (2022), no. 3, 981--1002.

\bibitem{FHHMZBanachSpaceTheory}
M. Fabian, P. Habala, P. H\'ajek, V. Montesinos, V. Zizler, \textit{Banach Space Theory. The Basis for Linear and Nonlinear Analysis}. Springer, New York, 2011.

\bibitem{Gamelin}
T. W. Gamelin, \textit{Complex Analysis}. Springer--Verlag, New York, 2001.

\bibitem{GLX}
N. Gao, D. H. Leung, F. Xanthos, \textit{Duality for unbounded order convergence and applications}. Positivity \textbf{22} (2018), no. 3, 711--725.

\bibitem{GHT}
E. Garc\'ia-S\'anchez, D. de Hevia, P. Tradacete, \textit{Free objects in Banach space theory}. Preprint \url{https://arxiv.org/abs/2302.10807}.

\bibitem{GhoussoubJohnson}
N. Ghoussoub, W. B. Johnson, \textit{Factoring operators through Banach lattices not containing $C(0, 1)$}. Math. Z. \textbf{194} (1987), 153--171.

\bibitem{dHT} 
D.~de~Hevia and P.~Tradacete, \textit{Free complex Banach lattices}. J. Funct. Anal. (in press).

\bibitem{JLTTT} 
H.~Jard\'on-S\'anchez, N.J.~Laustsen, M.A.~Taylor, P.~Tradacete, and V.G.~Troitsky, \textit{Free Banach lattices under convexity conditions}.  Rev. R. Acad. Cienc. Exactas F\'is. Nat. Ser. A Mat. RACSAM \textbf{116} (2022), no. 1, Paper no. 15.

\bibitem{Laust-Tra}
N. J. Laustsen and P. Tradacete, \textit{Banach lattices of homogeneous functions associated to a Banach space}. Preprint.  

\bibitem{LindenstraussTzafririVol2}
J. Lindenstrauss, L. Tzafriri, \textit{Classical Banach Spaces II}. Springer--Verlag, 1979.

\bibitem{MeyerNieberg}
P. Meyer-Nieberg, \textit{Banach Lattices}. Springer, Berlin, 1991.

\bibitem{Oikhberg}
T. Oikhberg, \textit{Geometry of unit balls of free Banach lattices, and its applications}. Preprint \url{https://arxiv.org/abs/2303.05209}.

\bibitem{OTTT} 
T. Oikhberg, M. A. Taylor, P. Tradacete, V. G. Troitsky, \textit{Free Banach lattices}. Preprint \url{https://arxiv.org/abs/2210.00614}.

\bibitem{dePW}
B.~de~Pagter and A.~W. Wickstead, \textit{Free and projective {B}anach lattices}. Proc. Roy. Soc. Edinburgh Sect. A \textbf{145} (2015), no.~1, 105--143.

\bibitem{Schaefer}
H.~H.~Schaefer, \textit{Banach Lattices and Positive Operators}. Springer--Verlag, 1974.
   
\bibitem{Talagrand} 
M. Talagrand, \textit{La structure des espaces de Banach r\'eticul\'es ayant la propri\'et\'e de Radon-Nikod\'ym.} Israel J. Math. \textbf{44} (1983), no. 3, 213--220. 

\end{thebibliography}
\end{document}